\documentclass[11pt,a4paper,final, dvipsnames]{amsart}

\input xy
\xyoption{all}

\DeclareMathAlphabet{\mathpzc}{OT1}{pzc}{m}{it}

\usepackage{latexsym,amsmath,amsfonts,amscd,amssymb}
\usepackage{graphics,color}

\usepackage[all]{xy}
\usepackage{graphicx}
\usepackage{stmaryrd}
\usepackage{enumitem}
\usepackage{comment}

\usepackage{empheq}
\usepackage{color}
\usepackage{xcolor}
\usepackage{bm}

\usepackage{hyperref}

\theoremstyle{plain}
\newtheorem{theorem}{Theorem}[section]
\newtheorem{corollary}[theorem]{Corollary}
\newtheorem{lemma}[theorem]{Lemma}
\newtheorem{proposition}[theorem]{Proposition}
\theoremstyle{definition}

\newtheorem{remark}[theorem]{Remark}

\newtheorem*{theorem*}{Theorem}
\newcommand{\red}{\mathrm{red}}
\newcommand{\irr}{\mathrm{irr}}

\newcommand{\g}{\gamma}

\newcommand{\G}{\Gamma}

\newcommand{\cD}{\mathcal{D}}

\newcommand{\cG}{\mathcal{G}}
\newcommand{\cH}{\mathcal{H}}
\newcommand{\cM}{\mathcal{M}}
\newcommand{\cI}{\mathcal{I}}

\newcommand{\cL}{\mathcal{L}}

\newcommand{\cR}{\mathcal{R}}

\newcommand{\cT}{\mathcal{T}}

\newcommand{\cV}{\mathcal{V}}

\newcommand{\RR}{\mathbb{R}}
\newcommand{\CC}{\mathbb{C}}

\newcommand{\ZZ}{\mathbb{Z}}

\newcommand{\x}{\times}

\newcommand{\la}{\langle}
\newcommand{\ra}{\rangle}
\newcommand{\frM}{{\frak M}}

\newcommand{\frX}{{\frak X}}

\newcommand{\Id}{\mathrm{Id}}
\newcommand{\id}{\mathrm{Id}}

\newcommand{\Sym}{\mathrm{Sym}}
\newcommand{\VarC}{\cV ar_k}

\DeclareMathOperator{\Hom}{Hom}
 \DeclareMathOperator{\diag}{diag}
 
\DeclareMathOperator{\Gr}{Gr}
\DeclareMathOperator{\GL}{GL}

\DeclareMathOperator{\PGL}{PGL}
\DeclareMathOperator{\SL}{SL}
\DeclareMathOperator{\SU}{SU}

\DeclareMathOperator{\Spec}{Spec}
\DeclareMathOperator{\tr}{tr}

\DeclareMathOperator{\PExp}{PExp}

\makeatletter
\@namedef{subjclassname@2020}{%
  \textup{2020} Mathematics Subject Classification}
\makeatother

 \title[Motive of the $\SL_4$-character variety of torus knots]{Motive of the $\SL_4$-character variety of torus knots}
 
 
 \author[\'A. Gonz\'alez-Prieto]{\'Angel Gonz\'alez-Prieto}
 \address{Departamento de
  \'Algebra, Geometr\'ia y Topolog\'ia, Universidad Complutense de Madrid, Plaza Ciencias 3, 28040 Madrid, Spain}
  \address{Instituto de Ciencias Matem\'aticas (CSIC-UAM-UCM-UC3M), C/ Nicol\'as Cabrera 13-15, 28049 Madrid, Spain}
  \email{angelgonzalezprieto@ucm.es}
  
  \author[V. Mu\~{n}oz]{Vicente Mu\~{n}oz}
 \address{Departamento de
  \'Algebra, Geometr\'ia y Topolog\'ia, Facultad de Ciencias,
  Universidad de M\'alaga, Campus de Teatinos s/n, 29071 Málaga,
  Spain}\email{vicente.munoz@uma.es}

\usepackage{amsmath}

\begin{document}

\begin{abstract}
  In this paper, we compute the motive of the character variety of representations of
  the fundamental group of the complement of an arbitrary torus knot into $\SL_4(k)$, for any
  algebraically closed field $k$ of zero characteristic. For that purpose, we introduce a stratification
  of the variety in terms of the type of a canonical filtration attached to any
  representation. % with semi-simple graded complex. 
  This allows us to reduce the computation of the motive to a combinatorial problem.
\end{abstract}

\maketitle

%%%%%%%%%%%%%%%%%%%%%%%%%%%%%%%%%%%%%%%%%%%%%%%%%%%%%%%%%%%%%%%%%%%%%%%%%%%%%%%%%%%%%%%
\section{Introduction}\label{sec:introduction}
%%%%%%%%%%%%%%%%%%%%%%%%%%%%%%%%%%%%%%%%%%%%%%%%%%%%%%%%%%%%%%%%%%%%%%%%%%%%%%%%%%%%%%%

Let $\Gamma$ be a finitely generated group and let $G$ be an algebraic group over an
algebraically closed field $k$ of zero characteristic. The space $R(\Gamma, G)$ of representations $\rho: \Gamma \to G$
forms an algebraic variety known as the representation variety. Moreover, if we want to parametrize
isomorphism classes of representations, we need to consider the Geometric Invariant Theory (GIT)
quotient
$$
	\frM(\Gamma, G) = R(\Gamma, G) \sslash G,
$$
where $G$ acts on $R(\Gamma, G)$ by conjugation. %, i.e.\ identifying isomorphic representations. 
This gives rise to the moduli space of representations of $\G$ into $G$, also called the $G$-character variety of $\G$.

Character varieties are very rich objects that contain subtle geometric information linking
distant areas in mathematics. An important instance is when we take $\Gamma = \pi_1(\Sigma)$
to be the fundamental group of the compact orientable surface of genus $g$. In this case, these character varieties are
one of the three incarnations of the moduli space of Higgs bundles, as stated by the celebrated
non-abelian Hodge correspondence \cite{Corlette:1988,hitchin,SimpsonI,SimpsonII}. For this reason, character varieties of surface groups
have been widely studied, particularly regarding the computation of some algebraic invariants like their
$E$-polynomial (an alternating sum of its Hodge numbers in the spirit of the %Poincar\'e polynomial and the 
Euler characteristic). Computing such invariant is a hard problem that has been tackled from an
arithmetic viewpoint \cite{Hausel-Letellier-Villegas:2013,Hausel-Rodriguez-Villegas:2008}, 
from a geometric perspective \cite{lomune, MM} and from the point of view of Topological Quantum Field Theories \cite{GP-2019, GPLM-2017}.

Character varieties also play a prominent role in the topology of $3$-manifolds, starting
with the foundational work of Culler and Shalen \cite{CS}. There, the authors used some quite simple algebro-geometric properties of $\SL_2(\CC)$-character varieties to provide new proofs of very remarkable results, as Thurston's theorem that
says that the space of hyperbolic structures on an acylindrical $3$-manifold is compact, or the Smith conjecture.

Since the publication of that paper, character varieties of $3$-manifolds have been the object of very intense research. In particular, they have been used to study knots $K\subset S^3$, by analyzing the character varieties associated to the fundamental group of their complements, $\Gamma = \pi_1(S^3-K)$. The geometry of these knot character varieties has been studied in \cite{Florentino-Lawton:2012,Florentino-Nozad-Zamora:2019,Lawton,LM} for trivial links (i.e.\ when $\Gamma$ is a free group) and \cite{Falbel:2014, HMP} for figure eight knots. However, one of the most studied cases is when $K$ is an $(n,m)$-torus knot. In this case, we need to consider representations of the group
$$
	\Gamma = \Gamma_{n,m} = \langle x, y \,|\, x^n = y^m \rangle
$$
into an algebraic group $G$, typically $G = \SL_r(\CC)$. For $G = \SL_2( \CC)$ they were studied in \cite{KM, Martin-Oller, Munoz}, and the case $G = \SL_3(\CC)$ was addressed in \cite{MP}. It was also studied for $G = \SU(2)$ in \cite{Martinez-Munoz:2015}.

Nevertheless, much less is known in the higher rank case $r \geq 4$, even for torus knots. In particular, providing a complete description of their geometry is still an open problem. In this paper, we will address the problem of computing, over an arbitrary algebraically
closed field $k$ 
of zero characteristic, the virtual classes of the character varieties of torus knots in the Grothendieck ring of algebraic varieties, $K (\VarC)$. As we will see, these classes lie in the subring of $K (\VarC)$ generated by the class of the affine line $q = [k]$ (the so-called Lefschetz motive), so it actually computes the motive of the character variety in the Chow ring and its $E$-polynomial. 

This paper is structured as follows. In Section \ref{sec:character} we review some generalities about representation and character varieties as well as the Grothendieck ring of algebraic varieties. Particularly, in Section \ref{sec:semi-simple-filt} we introduce a canonical filtration attached to a representation, with semi-simple graded representation, that we will refer to as the semi-simple filtration. This filtration will be very useful because it induces a natural decomposition of the representation variety according to the type (i.e.\ the dimension and multiplicity of each isotypic component of the graded pieces) of the semi-simple filtration of a representation.

In Section \ref{sec:torus}, we specialize to the case of torus knots. There, we outline the strategy that we will follow throughout this paper.
We focus on the subvariety $R^{\irr}(\G, \SL_r(k)) \subset R(\G, \SL_r(k))$ of irreducible representations. The key point is that, as shown in Proposition \ref{prop:cor:9.5}, an irreducible representation $\rho: \G_{n,m} \to \SL_r(k)$ lifts, up to rescaling, to a representation of the free product
\[
\begin{displaystyle}
   \xymatrix
   { \ZZ_n \star \ZZ_m \ar@{--{>}}[rd]^{\tilde{\rho}} & \\
    \Gamma_{n,m} \ar[r]_{\rho} \ar[u]  & \SL_r(k).
   }
\end{displaystyle}   
\]
This opens the door to computing the motive of the irreducible representation variety of torus knots through the representation variety of $\ZZ_n \star \ZZ_m$ via
$$
	[R^{\irr}(\G_{n,m}, \SL_r(k))] = [R(\ZZ_n \star \ZZ_m, \SL_r(k))]- [R^{\red}(\ZZ_n \star \ZZ_m, \SL_r(k))],
$$
where $R^{\red}(\ZZ_n \star \ZZ_m, \SL_r(k))$ is the subvariety of reducible representations of $\ZZ_n \star \ZZ_m$.

This greatly simplifies the problem because the representations $\rho: \ZZ_n \star \ZZ_m \to \SL_r(k)$ retain some crucial properties from the classical representation theory of finite groups, for instance, only finitely many configurations of eigenvalues are allowed for the elements $\rho(x), \rho(y) \in \SL_r(k)$. Hence, we can refine the decomposition of the representation variety according to its semi-simple filtration in order to obtain a decomposition
\begin{equation}\label{eq:red-decomp}
	R^{\red}(\ZZ_n \star \ZZ_m, \SL_r(k)) = \bigsqcup_{\kappa} \bigsqcup_{\tau \in \cT_\kappa^*} R(\tau),
\end{equation}
where $\kappa$ runs over the possible configurations of eigenvalues, $\cT_\kappa^*$ are the types of semi-simple filtrations that correspond to reducible representations with eigenvalues taken from $\kappa$, and $R(\tau)$ is the collection of representations $\rho: \ZZ_n \star \ZZ_m \to \SL_r(k)$ with semi-simple filtration of type $\tau \in \cT_\kappa^*$. Notice that the previous decomposition has finitely many strata since there are finitely many choices for $\kappa$ and $\tau$.

Leaving aside for a moment the count of the number of contributing strata, the problem of computing the motive $[R(\tau)] \in K (\VarC)$ for a fixed type $\tau$ is addressed in Section \ref{sec:fixed-type}, which is the heart of this paper. Roughly speaking, in Proposition \ref{prop:fibration} and Corollary \ref{cor:no-action-discrete} we will show that (for low rank) $R(\tau)$ is the total space of a locally trivial fibration in the Zariski topology
\begin{equation}\label{eq:fibration-Xtau}
	R(\tau) \to \frM_\tau,
\end{equation}
that fibers over the subvariety $\frM_\tau$ of semi-simple representations of type $\tau$. The fiber of this map is given by a quotient $\cM_\tau / \cG_\tau$, where $\cM_\tau$ is a variety parametrizing the ways in which a semi-simple representation can be completed to an arbitrary reducible representation and $\cG_\tau \subset \GL_r(k)$ is a \emph{gauge group} acting freely on $\cM_\tau$, that identifies equivalent completions.

The description of Section \ref{sec:fixed-type} will be enough for computing the motives of $\cM_\tau$ and $\cG_\tau$. In this way, in Section \ref{sec:explicit-formulas} we give a combinatorial recipe for calculating them in terms of the repeated eigenvalues of $\kappa$ and the type $\tau$. Observe that, since we are dealing with reducible representations, all the irreducible components appearing in $\frM_\tau$ are forced to have lower rank. In such manner, the results of this section provide a recursive method for computing the motive of $R(\tau) \subset R(\G_{n,m}, \SL_r(k))$ from the knowledge of $[R^{\irr}(\G_{n,m}, \SL_s(k))]$ for $s < r$.

Regarding the number of components in (\ref{eq:red-decomp}), as a byproduct of Section \ref{sec:explicit-formulas}, we will show that the motive $[R(\tau)]$ for $\tau \in \cT_\kappa^*$ does not depend on the particular values of the eigenvalues of the configuration $\kappa$, but only on the multiplicity of its repeated eigenvalues. Hence, the sum in $\kappa$ of (\ref{eq:red-decomp}) amounts to multiplying $[R(\tau)]$ by a combinatorial coefficient that counts the number of admissible configurations with prescribed multiplicities. This counting is a combinatorial problem that is addressed in Section \ref{sec:count-comp}. In Section \ref{subsec:coprime}, we perform the calculation in the case that the (rather artificial) condition $\gcd(n,r) = \gcd(m,r) = 1$ holds, obtaining closed formulas for these coefficients in terms of the usual multinomial numbers. This hypothesis is dropped in Section \ref{sec:number-components}, where the general case is studied by means of generating functions. In that setting, we get more involved formulas that turn out to agree with the simpler closed forms of Section \ref{subsec:coprime} for $r \leq 4$, even if the  primality condition no longer holds. 

In Section \ref{sec:rank-low}, we use the method developed in this paper to compute the motive of the character variety of irreducible representations of torus knots for $G = \SL_2(k)$ (Section \ref{sec:rank2}) and $G = \SL_3(k)$ (Section \ref{sec:rank3}). This reproves the existing calculations in the literature of these motives in \cite{Munoz} and \cite{MP}, respectively. For that purpose, we provide a detailed description of the possible types and configurations of eigenvalues that may occur. 

The aim of Section \ref{sec:rank4} is to address the computation of the character variety of irreducible representations in the case $G = \SL_4(k)$, $\frM^{\irr}(\Gamma, \SL_4(k))$, which is the main novel contribution of this paper. This rank is the first case in which we need to deal with representations with isotypic components of dimension $2$ and multiplicity higher than $1$. For this reason, we divide the analysis into two parts. In the first case, studied in Section \ref{subsec:8.1}, we consider the seven eigenvalue configurations in which such high multiplicity components cannot happen. The counting in this case is completely analogous to the rank $2$ and $3$ cases.

On the other hand, in Section \ref{subsec:8.2} we deal with the three eigenvalue configurations $\kappa$ that admit repeated isotypic components. In this case, the calculation of $\cM_\tau$ for a type $\tau \in \kappa$ is trickier than in Section \ref{subsec:8.1} since, roughly speaking, $\cM_\tau$ is the complement of some Schubert cells as described in Section \ref{sec:fixed-type}. Nonetheless, this more involved situation can be analyzed with the tools introduced in Section \ref{sec:explicit-formulas}. Due to the huge number of possible types that we need to consider (more than $350$ in comparison with $2$ for $G = \SL_2(k)$ and $23$ for $G = \SL_3(k)$), we do not provide an exhaustive enumeration of all the possibilities. Nevertheless, the description of Sections \ref{sec:fixed-type} and \ref{sec:explicit-formulas} is explicit enough to give rise to a precise algorithm that can be implemented in a computer algebra system. An implementation in SageMath \cite{SageMath} created by the authors can be checked in \cite{GPMScript}. From this computation, we obtain the main result of this paper.

\begin{theorem} \label{thm:main}
The motive of the irreducible $\SL_4(k)$-character variety of the group
$\Gamma = \Gamma_{n,m}$ of the $(n,m)$-torus knot is given by
\begin{equation*}
\begin{aligned}
[\frM^{\irr}(&\Gamma, \SL_4(k))] \!\!=\! {\textstyle \frac{4}{nm} \binom{n}{4}\binom{m}{4} }
{\left(q^{9} + 6  q^{8} + 20  q^{7} + 17  q^{6} - 98  q^{5} - 26  q^{4} + 38  q^{3} + 126  q^{2} - 144\right)}  \\
&+ {\textstyle \frac{4}{nm} \Big(\binom{n}{4}\binom{m}{2,1}+ \binom{n}{2,1}\binom{m}{4}\Big)}
{\left(q^{7} + 5  q^{6} + 7  q^{5} - 34  q^{4} + 34  q^{2} + 18  q - 48\right)}  \\
&+ {\textstyle \frac{4}{nm} \Big(\binom{n}{4}\binom{m}{2}+ \binom{n}{2}\binom{m}{4}\Big)}
 {\left(q^{5} + 4  q^{4} - 11  q^{3} + q^{2} + 18  q - 18\right)}  \\
&+ {\textstyle \frac{4}{nm} \Big(\binom{n}{4}\binom{m}{1,1}+ \binom{n}{1,1}\binom{m}{4}\Big)}
 {\left(q^{3} - 4  q^{2} + 6  q - 4\right)}  \\ & + 
 {\textstyle \frac{4}{nm} \Big(\binom{n}{2,1}\binom{m}{2}+ \binom{n}{2}\binom{m}{2,1}\Big)}
 {\left(q^{3} - 3  q^{2} + 5  q - 4\right)} \\
&+ {\textstyle \frac{4}{nm} \binom{n}{2,1}\binom{m}{2,1}}
 {\left(q^{5} + 2  q^{4} - 10  q^{3} + 7  q^{2} + 11  q - 17\right)},
\end{aligned}
\end{equation*}
where $q = [k] \in K (\VarC)$ is the motive of the affine line and the multinomial numbers are defined in Corollary \ref{cor:form-N-final}.
\end{theorem}

In Section \ref{sec:motive-total}, we move from the motive of the moduli space of irreducible representations, $\frM^{\irr}(\Gamma, \SL_4(k))$, to the motive of the total character variety $\frM(\Gamma, \SL_4(k))$. For that purpose, we use the well-known fact that $\frM(\Gamma, \SL_4(k))$ parametrizes isomorphism classes of semi-simple representations and, therefore, the calculation can be reduced to combinations of the previous results for the character varieties of irreducible representations for all the ranks $1 \leq r \leq 4$. 

We end this section with some final words about the higher rank case $r > 4$. The results proved in this paper about the structure of the irreducible representations and their stratification according to the type of their semi-simple filtration are completely general and may be used in higher rank. However, two issues prevent the method to directly generalize to these cases. The first one is purely computational, since the number of types to be analyzed grows exponentially with rank, so to consider higher rank there is an extrinsic limitation imposed by the available computational power. The second one is much subtler and is hidden in Section \ref{sec:fixed-type} and Proposition \ref{prop:fibration}. The key point is that, more precisely, in general the fibration (\ref{eq:fibration-Xtau}) is not locally trivial in the Zariski topology. In the general case, this fibration presents non-trivial monodromy coming from the action on $\frM_\tau$ by permuting isomorphic irreducible representations. In ranks $r \leq 4$, the only repeated irreducible representations that may occur in $\frM_\tau$ are $1$-dimensional (we get rid of the case of two $2$-dimensional representations in Proposition \ref{prop:8.1}) and, since the moduli space of such representations is just a point, the action by permutations is trivial.

Despite these difficulties, the approach developed in this paper can be used to tackle the general rank case. A more involved analysis of this permutation action may be performed in order to understand the monodromy action, 
and from it the motive of the total space may be computed as in \cite{lomune}. To this extend, our approach is suitable for studying the higher rank case.

The results of this paper also hold for an algebraically closed field of positive characteristic $p>0$,
as long as $n,m,r$ are coprime with
$p$. The relevant point is that we take roots of unity and we count them.

\noindent \textbf{Acknowledgements.} The authors want to thank Eduardo Fern\'andez-Fuertes and Marina Logares for very useful conversations around the methods of this paper. We specially thank Carlos Florentino and Sean Lawton for pointing out some errors in a previous version of this manuscript and for references. The second author is partially supported by Project MINECO (Spain) PGC2018-095448-B-I00. 

%%%%%%%%%%%%%%%%%%%%%%%%%%%%%%%%%%%%%%%%%%%
\section{Moduli of representations and character varieties}\label{sec:character}
%%%%%%%%%%%%%%%%%%%%%%%%%%%%%%%%%%%%%%%%%%%

Fix an algebraically closed field $k$ of zero characteristic. Let $\G$ be a finitely generated group, and let $G$ be a reductive algebraic
group over $k$. A \textit{representation} of $\G$ in $G$ is a homomorphism $\rho: \G\to G$.
Consider a presentation $\G=\la \g_1,\ldots, \g_k | \{r_{\lambda}\}_{\lambda \in \Lambda} \ra$, where $\Lambda$ is the (possibly infinite) indexing set of relations of $\G$. Then $\rho$ is completely
determined by the $k$-tuple $(A_1,\ldots, A_k)=(\rho(\g_1),\ldots, \rho(\g_k))$
subject to the relations $r_\lambda(A_1,\ldots, A_k)=\id$, for all $\lambda \in \Lambda$. The space
of representations is
 \begin{eqnarray*}
 R(\G,G) &=& \Hom(\G, G) \\
  &=& \{(A_1,\ldots, A_k) \in G^k \, | \,
 r_\lambda(A_1,\ldots, A_k)=\id \}\subset G^{k}\, .
 \end{eqnarray*}
Therefore $R(\G,G)$ is an affine algebraic set.

\begin{remark}Even though $\Lambda$ may be an infinite set, without loss of generality we can suppose that $ R(\G,G)$ is defined by finitely many equations. Indeed, consider the collection of ideals generated by finitely many of these relations, partially ordered through inclusion. Since the coordinate ring of $G^k$ is Noetherian, such a collection has a maximal element, say, the ideal generated by the relations $r_{\lambda_1}-\Id, \ldots, r_{\lambda_s}-\Id$. This ideal coincides with the ideal generated by the whole set of relations since, otherwise, we could enlarge the set $r_{\lambda_1}-\Id, \ldots, r_{\lambda_s}-\Id$, contradicting the maximality.
%Indeed, the ideal generated by the whole set of relations $\{r_\lambda -1 \}_{\lambda \in \Lambda}$ of the coordinate ring of $G^k$ is finitely generated by noetherianity, so taking the ideal generated by those relations appearing in the set of generators we end up with an ideal generated by finitely many relations $r_{\lambda_1}-1, \ldots, r_{\lambda_s}-1$.
\end{remark}

We say that two representations $\rho$ and $\rho'$ are
equivalent if there exists $g \in G$ such that $\rho'(\gamma)=g^{-1} \rho(\gamma) g$,
for every $\gamma \in \G$.
%Fix a representation $G \subset \GL_r(k)$ as a linear group. Under this embedding, the equivalence relation corresponds to a change of basis in the $r$-dimensional linear space $V=k^r$.
This produces a natural action of $G$ in $R(\G,G)$ by conjugation. The moduli space of representations
is the GIT quotient
 $$
 \frM(\G,G) = R(\G,G) \sslash G \, .
 $$
Recall that by definition of GIT quotient for an affine variety, if we write
$ R(\G,G)=\Spec A$, then 
$\frM (\G,G)=\Spec A^{G}$, where $A^G$ is the finitely generated $k$-algebra of invariants elements
of $A$ under the induced action of $G$.

From now on, let us focus on the groups $G= \GL_r(k)$ or $G = \SL_r(k)$. In this setting, a representation $\rho$ is said to be \textit{reducible} if there exists some proper linear
subspace $W\subset V$ such that for all $\g \in \G$ we have 
$\rho(\g)(W)\subset W$; otherwise $\rho$ is
\textit{irreducible}. This gives subsets 
 $$
 R^{\red}(\G, G), R^{\irr}(\G, G) \subset R(\G, G)
 $$ 
of reducible and irreducible
representations, that are closed and open subvarieties, respectively. The action of $G$ on the whole $R(\G, G)$ might be quite
involved. However, if $Z(G)$ denotes the center of $G$, which acts trivially on $R(\G, G)$, then the induced action of $G/Z(G)=\PGL_r(k)$ 
on $R^{\irr}(\G, G)$ is free by Schur lemma.

Note that if $\rho$ is reducible, then let $W \subset V$ be
an invariant subspace (of some dimension $s$), and consider a complement $V =W\oplus W'$. 
Let $\rho_1=\rho|_W$ and let $\rho_2$ be the induced representation on $W'$. Then we can write $\rho=\begin{pmatrix} \rho_1 & f\\
0& \rho_2\end{pmatrix}$, where $f: \G \to \Hom(W',W)$. Take $P_t=
\begin{pmatrix} t^{s-r}\,\Id & 0\\ 0& t^{s}\,\Id \end{pmatrix}$. %, where $k=\dim V$.
Then $P_t^{-1}\rho P_t=\begin{pmatrix} \rho_1 & t^{r}f \\
0 & \rho_2\end{pmatrix} \to \hat\rho=\begin{pmatrix} \rho_1 & 0\\
0& \rho_2\end{pmatrix}$, when $t\to 0$. Therefore $\rho$ and $\hat\rho$
define the same point in the quotient $\frM(\G,G)$. Repeating this, we can
substitute any representation $\rho$ by some $\hat\rho=\bigoplus \rho_i$,
where all $\rho_i$ are irreducible representations. We call this process 
\emph{semi-simplification}, and $\hat\rho$ is called a semi-simple
representation; also $\rho$ and $\hat\rho$ are called
S-equivalent. The space $\frM(\G,G)$ parametrizes semi-simple representations
\cite[Thm.~ 1.28]{LuMa}. Analogously, we can also consider the moduli space of irreducible representations $\frM^{\irr}(\G, G) = R^{\irr}(\G, G) \sslash G$.

Suppose now that $k=\CC$. Given a representation $\rho: \G\to G$, we define its
\textit{character} as the map $\chi_\rho: \G\to \CC$,
$\chi_\rho(g)=\tr \rho (g)$. Note that two equivalent
representations $\rho$ and $\rho'$ have the same character.
There is a character map $\chi: R(\G,G)\to \CC^\G$, $\rho\mapsto
\chi_\rho$, whose image
 $$
 \frX(\G,G)=\chi(R(\G,G))
 $$
is called the \textit{character variety of $\G$}. Let us give
$\frX(\G,G)$ the structure of an algebraic variety. The traces $\chi_\rho$
span a subring $B\subset A$. Clearly $B\subset A^{G}$, and it can be proven that
$B$ is actually a finitely generated $\CC$-algebra \cite{LuMa}. Hence
there exists a collection $\g_1,\ldots, \g_a$ of elements of
$\G$ such that $\chi_\rho$ is determined by $\chi_\rho(\g_1),\ldots,
\chi_\rho(\g_a)$, for any $\rho$. Such a collection gives a map
 $\bar\chi :R(\G,G)\to \CC^a$, $\bar\chi(\rho)=(\chi_\rho(\g_1),\ldots, \chi_\rho(\g_a))$,
and $\frX(\G,G)\cong \bar\chi(R(\G,G))$. This endows $\frX(\G,G)$ with
the structure of an algebraic variety, which is independent of the chosen collection.
The natural algebraic map  
 $$
 \frM(\G,G)\to \frX(\G,G)
 $$ 
is an  isomorphism for $G = \SL_{r}(\CC)$, see \cite[Chapter 1]{LuMa}. This is the same as to say that 
$B=A^G$, that is, the ring of invariant polynomials is generated by characters. However, for other reductive groups this map may not be an isomorphism, as for $G = \mathrm{SO}_2(\CC)$ \cite[Appendix A]{Florentino-Lawton:2012}. For a general discussion on this issue, see \cite{Lawton-Sikora:2019}.

%%%%%%%%%%%%%%%%%%%%%%%%%%%%%%%%%%
\subsection{The semi-simple filtration}\label{sec:semi-simple-filt}
%%%%%%%%%%%%%%%%%%%%%%%%%%%%%%%%%%

Throughout this section, we will fix $G = \SL_r(k)$, for a fixed $r > 0$ and $k$ an algebraically closed field of zero characteristic, and $\Gamma$ a
finitely generated group. The conjugacy action of $\SL_r(k)$ on $R(\Gamma, \SL_r(k))$ extends to an action
of $\GL_r(k)$. Moreover, in order to shorten the notation, when the underlying field is understood we will denote
$\SL_r = \SL_r(k)$ and $\GL_r = \GL_r(k)$.

A very important feature of representations (indeed, of objects of an abelian category) is that they have attached
a canonical filtration, called the semi-simple filtration.

\begin{proposition}
Let $V$ be a representation of $\Gamma$. There exists an unique filtration of $\Gamma$-modules
$$
	0 = V_0 \subset V_1 \subset \ldots \subset V_i \subset \ldots \subset V_s = V,
$$
such that $\Gr_i(V_\bullet) = V_{i}/V_{i-1}$ is a maximal semi-simple subrepresentation of $V/V_{i-1}$.
\end{proposition}

\begin{proof}
The existence is trivial, just take $V_1 \subset V$ to be a maximal semi-simple subrepresentation and repeat
inductively on $V/V_1$. For the uniqueness, suppose that $V_\bullet$ and $V'_\bullet$ are two such filtrations.
Then, we must have $V_1 = V'_1$ and working inductively the result follows. Indeed, write $V_1 = W_1 \oplus \ldots \oplus W_s$
and $V'_1 = W'_1 \oplus \ldots \oplus W'_{s'}$. Reordering if necessary,
suppose that $W_j = W'_j$ for $j = 1, 2, \ldots, t$, and $V_1\cap V_1'=W_1\oplus \ldots\oplus W_t$. 
Then $W_1 \oplus \ldots \oplus W_t \oplus W_{t+1}\oplus \ldots \oplus W_s\oplus W'_{t+1}
\oplus \ldots \oplus W'_{s'}$ is a semi-simple subrepresentation of $V$ (the sum is direct) containing $V_1$. By
maximality, it is equal to $V_1$, hence $s=s'$ and $V_1=V_1'$.
\end{proof}

Now, let $V$ be a $\Gamma$-representation and let $V_\bullet$ be its semi-simple filtration. Write the graded pieces into its
isotypic components 
 \begin{equation}\label{eqn:Wij}
 \Gr_i(V_\bullet) \cong \bigoplus_{j=1}^{s_i} W_{i,j}^{m_{i,j}}\, ,
 \end{equation}
with $W_{i,1}, \ldots, W_{i,s_i}$ non-isomorphic representations. These are called \emph{isotypic components} and
$m_{i,j}$ is their multiplicities. From this information, we can define the \emph{shape} of the representation as the tuple
\begin{align*}
	\xi &= \Big(\left\{(\dim W_{1,1}, m_{1,1}), \ldots, (\dim W_{1,r_1}, m_{1,r_1}) \right\}, \ldots, \\
	& \qquad \left\{(\dim W_{s,1},
	m_{s,1}), \ldots, (\dim W_{s,r_s}, m_{s,r_s})) \right\}\Big).
\end{align*}
Moreover, we can add spectral information to the shape. For each $\gamma \in \Gamma$, we denote by
$\sigma_{i,j}(\gamma) = \big\{\epsilon_{i,j,1}^{n_1}, \ldots, \epsilon_{i,j,l}^{n_l}\big\}$ the \emph{collection of eigenvalues}
(as a multiset with multiplicities, where the exponent means the number of times that the eigenvalue is repeated) 
of the action of $\gamma$ on $W_{i,j}$. This gives
 $$
 \sigma=(\sigma_{i,j}(\gamma) )_{\gamma\in \Gamma}
 $$
and the pair 
 $$
 \tau = (\xi, \sigma)
 $$ 
will be called the \emph{type} of the representation.

In such manner, the types of the representations induce a decomposition of the representation variety as
\begin{equation}\label{eq:strat-types}
	R(\Gamma, \SL_r) = \bigsqcup_\tau R(\tau),
\end{equation}
where $R(\tau)$ is the set of $\Gamma$-representations of type $\tau$. Observe that this decomposition is compatible with the 
action of $\GL_r$ by conjugation.

\begin{remark}\label{rmk:type-irr}
The semi-simple representations are the representations whose type has a single step, i.e.\ $s=1$. Furthermore, the irreducible
 representations correspond to the shape $\xi = (\left\{(r, 1)\right\})$, that is those whose only graded piece has a single 
 isotypic component of multiplicity $1$.
\end{remark}

%%%%%%%%%%%%%%%%%%%%%%%%%%%%%%%%%%%%%%%
\subsection{Grothendieck ring of varieties} \label{sec:KVar}
%%%%%%%%%%%%%%%%%%%%%%%%%%%%%%%%%%%%%%%

Let $\VarC$ be the category of quasi-projective varieties over $k$.
We denote by $K (\VarC)$ the \emph{Grothendieck ring} of
$\VarC$. This is the abelian group generated by elements $[Z]$, for
$Z \in \VarC$, subject to the relation $[Z]=[Z_1]+[Z_2]$ whenever $Z$
can be decomposed as a disjoint union $Z=Z_1\sqcup Z_2$ of a closed and
a Zariski open subset. The element $[Z]$ in $\VarC$ associated to a variety $Z$ is
called a \emph{motive}. 

There is a naturally defined product in $K (\VarC)$ given by $[Y]\cdot
[Z]=[Y\x Z]$. Note that if $\pi:Z\to Y$ is an algebraic fiber
bundle with fiber $F$, which is locally trivial in the Zariski
topology, then $[Z]=[F]\cdot [Y]$.

We denote by $q=[k]$ the \emph{Lefschetz motive} in $K (\VarC)$. 
We have cases that will be used later on:
 \begin{itemize}
 \item $[k^r]=q^r$,
 \item $[\GL_r]=(q^r-1)(q^r-q) \cdots (q^r-q^{r-1})$,
 \item $[\SL_r]= [\PGL_r] =(q^r-1)(q^r-q) \cdots (q^r-q^{r-2})q^{r-1}$.
\end{itemize}

%%%%%%%%%%%%%%%%%%%%%%%%%%%%%%%%%%%%%%%%%%%%%%%%%%%%%%%%
\section{Representation varieties of torus knots} \label{sec:torus}
%%%%%%%%%%%%%%%%%%%%%%%%%%%%%%%%%%%%%%%%%%%%%%%%%%%%%%%%

Let $T^2=S^1 \times S^1$ be the $2$-torus and consider the standard embedding
$T^2\subset S^3$. Let $n,m$ be a pair of coprime positive integers. Identifying
$T^2$ with the quotient $\RR^2/\ZZ^2$, the image of the straight line $y=\frac{n}{m}
x$ in $T^2$ defines the \textit{torus knot} of type $(n,m)$, which we shall denote
as $K_{n,m}\subset S^3$ (see \cite[Chapter 3]{Rolfsen}). It is known that the fundamental group of the exterior of the torus knot is
 \begin{equation}\label{eqn:torus}
  \G_{n,m}= \pi_1({S^3- K_{n,m}}) \cong \la x,y \, | \, x^n= y^m \,\ra \,.
 \end{equation}
 
The aim of this paper is to describe the representation variety $R(\G_{m,n}, \SL_r(k))$, for $k$ an algebraically closed field of zero characteristic.
%For this purpose, we also need the variety $X(\G_{m,n}, \GL_r(k))$. 
We introduce the notation
 \begin{equation}\label{eqn:cXr}
 \cR_r = R(\G_{n,m}, \SL_r),  %\quad \tilde\cR_r = X(\G_{n,m}, \GL_r), 
 \end{equation}
dropping the reference to $n,m$. Clearly $ \cR_1$ consists of one point. 
%In \cite[Lemma 4.1]{MP}, it was proved that 
% $\tilde \cR_1\cong k^* = k - \left\{0\right\}$. 
We will denote the irreducible and reducible subvarieties of representations by %of $\SL_r$ (resp.\ $\GL_r$) by
 $ \cR_r^{\irr}$ and $ \cR_r^{\red}$. % (resp.\ $\tilde \cR_r^{\irr}$ and $\tilde \cR_r^{\red}$). Observe that a representation
 A representation $\rho: \Gamma_{m,n} \to \SL_r$ is completely determined by a pair of matrices $(A,B) = (\rho(x), \rho(y))$.

\begin{lemma}\label{lem:red}
Suppose that $\rho=(A,B)\in \cR_r^{\irr}$. Then $A^n=B^m=\varpi\, \Id$, for 
some $\varpi$ an $r$-th root of unit.
\end{lemma}

\begin{proof}
Set $P = A^n = B^m$. We have that $PA = A P$ and $PB = B P$, so $P^{-1} \rho P = \rho$. 
Thus, $P$ is a $\Gamma$-equivariant automorphism of $\rho$ which, by Schur lemma, implies that 
$P = \varpi\, \Id$ for some $\varpi \in k^*$. If $(A,B)\in \cR_r$, we must have $\varpi^r = \det (\varpi \, \Id) = 1$.
\end{proof}

We denote by $\mu_r$ the set of $r$-th roots of unity.
Given $\varpi \in \mu_r$, denote
 $$
 \cR_r(\varpi)= \{ (A,B) \in \cR_r  \,|\, A^n=B^m=\varpi \, \Id\}.
 $$
Lemma \ref{lem:red} implies that there is a stratification (into disjoint subsets)
 \begin{equation}\label{eqn:yyy}
 \cR_r=\bigsqcup_{\varpi\in \mu_r} \cR_r(\varpi).
 \end{equation}

We also introduce the group $\ZZ_n \star \ZZ_m$, its representation variety 
$R(\ZZ_n \star \ZZ_m, \GL_r)$, and the sets 
 $$
 R_r= R_r^{(\xi_1,\xi_2)} = \{(A,B)\in R(\ZZ_n \star \ZZ_m, \GL_r) | \det A=\xi_1,\det B=\xi_2\},
  $$
for $\xi_1,\xi_2\in k^*$ (we usually drop $\xi_1,\xi_2$ from the notation). 
If we change $\xi_1,\xi_2$ to $\xi_1'=\lambda_1\xi_1$, $\xi_2'=\lambda_2\xi_2$, where
$\lambda_1^{1/r}\in \mu_n$, $\lambda_2^{1/r}\in \mu_m$, the variety $R_r$ remains the same (isomorphic).

\begin{proposition} \label{prop:cor:9.5} 
There is an isomorphism
$$
 \cR_r^{\irr}(\varpi) \cong R_r^{\irr} \subset R_r^{(\xi_1,\xi_2)}\, ,
 $$
for $\xi_1=\varpi^{-r/n}$, $\xi_2=\varpi^{-r/m}$. 
\end{proposition}

\begin{proof}
Since $A^n = B^m = \varpi\, \Id$, if we fix $n$-th and $m$-th roots of $\varpi^{-1}$, 
denoted by $\eta_1$ and $\eta_2$ respectively, then we have
that $(\eta_1A)^n = (\eta_2B)^m = \Id$. Hence, the pair $(\eta_1A, \eta_2 B)$ defines a representation of $R(\ZZ_n \star \ZZ_m, \GL_r)$.
Now, observe that $\det (\eta_1A) = \eta_1^r \det(A) = \eta_1^r = \xi_1$, and similarly $\det (\eta_2B) = \xi_2$, so the pair $(\eta_1A, \eta_2 B)$ actually lies in $R_r^{(\xi_1,\xi_2)}$.
Clearly, the image of this assignment $(A, B) \mapsto (\eta_1A, \eta_2 B)$ coincides with the subset of irreducible representations.

%Take $\eta_1$ and $\eta_2$ such that $\eta_1^n = \eta_2^m = \varpi^{-1}$. 
%Then the map $(A,B)\mapsto (A',B')=(\eta_1 A,\eta_2B)$ sends a representation
%of $\cR_r^{\irr}(\varpi)$ to a pair $(A',B')\in R(\ZZ_n \star \ZZ_m, \GL_r)$
%since $(A')^n=\eta_1^n A^n=\Id$ and $(B')^m=\eta_2^m B^m=\Id$. Also
%$\det A=\eta_1^r$, $\det B=\eta_2^r$, so the image of this map lies
%in  $R_r^{(\xi_1,\xi_2)}$. Clearly it coincides with the subset of
%irreducible representations.
\end{proof}

\begin{corollary}
If $\rho=(A,B)\in \cR_r^{\irr}$, then $A$ and $B$ are diagonalizable.
\end{corollary}

\begin{proof}
Any representation of a finite group is completely reducible and, if the group is abelian, the irreducible representations
are $1$-dimensional.
\end{proof}

\begin{remark}
$A$ and $B$ are diagonalizable, but they cannot be diagonalizable in the same basis. Indeed, they cannot share even an eigenvector 
since, otherwise, the representation $\rho$ would be reducible.
\end{remark}

By Proposition \ref{prop:cor:9.5}, we focus on $R_{r}$.
  Observe that the possible eigenvalues of $(A,B) \in R_{r}$ form a discrete set, since $A^n=\Id$ implies that the eigenvalues of $A$
 are $n$-th roots of unity, and similarly for $B$.
 This allows a refinement
 of the stratification of (\ref{eq:strat-types}). Let $\rho = (A,B) \in R_{r}$ and let $\bm{\epsilon} = \big\{\epsilon_1^{a_1},
  \ldots, \epsilon_p^{a_p}\big\}$ and $\bm{\varepsilon}=\big\{\varepsilon_1^{b_1}, \ldots, \varepsilon_q^{b_q}\big\}$ be 
  the eigenvalues of $A$ and $B$ respectively (as multisets with repetitions), that will be collected in the \emph{configuration of 
  eigenvalues} 
  \begin{equation}\label{eqn:kappa}
  \kappa = (\bm{\epsilon}, \bm{\varepsilon}).
  \end{equation}
This implies that $R_{r}$ can be decomposed as a finite disjoint  union
\begin{equation}\label{eq:dec-irred-kappa}
 R_{r} = \bigsqcup_\kappa R_{\kappa}, \quad \textrm{where }\, \,  R_{\kappa} = \bigsqcup_{\tau \in \cT_\kappa} R(\tau).
\end{equation}
Here $R(\tau) \subset R_{r}$ is, as in (\ref{eq:strat-types}), the collection of representations of $R_{r}$ of type 
$\tau=(\xi,\sigma)$, and $\cT_\kappa$ is the collection of types whose configuration of eigenvalues $\sigma$ is taken from $\kappa$.

From this decomposition, we can understand the virtual class of $\cR_r^{\irr}(\varpi) \subset R_{r}$. Observe that, if 
$\cT_\kappa^* \subset \cT_\kappa$ are the types whose shape does not correspond to an irreducible representation, then 
$R_{\kappa}^{\red} = \bigsqcup\limits_{\tau \in \cT_\kappa^*} R(\tau)$ is the collection of reducible representations of $R_{\kappa}$,
and the irreducible representations are $R_{\kappa}^{\irr} = R_{\kappa}-R_{\kappa}^{\red}$. Hence, we have that
 \begin{equation}\label{eqn:working-eqn1}
	\cR_r^{\irr}(\varpi) = R_r^{\irr} =
	\bigsqcup_{\kappa} R_{\kappa}^{\irr} = \bigsqcup_{\kappa} \left(R_{\kappa} - R_{\kappa}^{\red}\right) 
	= \bigsqcup_{\kappa} \Big(R_{\kappa} - \bigsqcup_{\tau \in \cT_\kappa^*} R(\tau)\Big).
 \end{equation}

The virtual class $[R_{\kappa}]$ is very easy to compute. We have that
$$
	R_{\kappa} = \left(\GL_r \cdot \Sigma_{\bm{\epsilon}}\right) \times \left(\GL_r \cdot \Sigma_{\bm{\varepsilon}}\right),
$$
where $\Sigma_{\bm{\epsilon}}$ is the diagonal matrix with eigenvalues $\bm{\epsilon}$ and $\GL_r\cdot  \Sigma_{\bm{\epsilon}}$ is 
its orbit under the $\GL_r$-action by conjugation, and analogously for $\bm{\varepsilon}$. The stabilizers of $\Sigma_{\bm{\epsilon}}$ 
and $\Sigma_{\bm{\varepsilon}}$ under this action are $\GL_{a_1} \times \ldots \times \GL_{a_p}$ and $\GL_{b_1} \times \ldots 
\times \GL_{b_q}$, respectively. Therefore, we have that
  \begin{equation}\label{eqn:working-eqn2}
	\left[R_{\kappa}\right] = \frac{[\GL_r]}{[\GL_{a_1}] \cdots [\GL_{a_p}]} \cdot \frac{[\GL_r]}{[\GL_{b_1}] \cdots [\GL_{b_q}]}.
 \end{equation}
 
Using (\ref{eqn:working-eqn1}), we get that, in order to compute the virtual class of $\cR_r^{\irr}(\varpi)$, it is enough to compute the 
virtual classes of $R(\tau)$, for all $\kappa$ and $\tau \in \cT_\kappa^*$. As we will show in the following section, this 
amounts to a combinatorial problem and the knowledge of $[\cR_s^{\irr}(\varpi)]$ for $s < r$, so the computation can be 
performed recursively.

From this computation, one can also compute the virtual class of $\frM^{\irr}_r = R_r^{\irr} \sslash \SL_r$. Indeed, the action of $\PGL_r = \SL_r/Z(\SL_r)$ is free and has closed orbits on the irreducible representations. Hence, working analogously to Proposition 7.3 of \cite{GP:2018} for arbitrary rank, we get that
$$
	[\frM^{\irr}_r] = \frac{[R_r^{\irr}]}{[\PGL_r]}.
$$
This action respects the stratification (\ref{eq:dec-irred-kappa}) so defining $\frM_\kappa^{\irr} = R_\kappa^{\irr} \sslash \SL_r$ we have
\begin{equation}\label{eq:m-kappa}
	[\frM^{\irr}_r] = \sum_\kappa [\frM^{\irr}_\kappa] = \sum_\kappa \frac{[R_\kappa^{\irr}]}{[\PGL_r]}.
\end{equation}

\begin{remark}\label{rmk:reducible}
Consider a configuration of eigenvalues $\kappa = (\bm{\epsilon}, \bm{\varepsilon})$ with multiplicities $a_i$ and $b_j$ 
respectively. If there exists $i,j$ such that $a_i + b_j > r$, then $R_\kappa^{\irr} = \emptyset$. Indeed, if $\rho = (A, B)
 \in R_\kappa$, then $A$ has an eigenspace $V$ of dimension $a_i$, and $B$ has another eigenspace $W$ of dimension $b_j$. But since
  $a_i + b_j > r$, $V \cap W \neq 0$ and thus, there exists at least a common eigenvector to $A$ and $B$ so they are reducible.
  Hence we can discard such $\kappa$ in the formula (\ref{eqn:working-eqn1}).
\end{remark}

%%%%%%%%%%%%%%%%%%%%%%%%%%%%%%%%%%%%%%%%%%%%%%
\section{Representations of fixed type}\label{sec:fixed-type}
%%%%%%%%%%%%%%%%%%%%%%%%%%%%%%%%%%%%%%%%%%%%%%

Throughout this section, we will fix a configuration of eigenvalues $\kappa = (\bm{\epsilon}, \bm{\varepsilon})$ and a type $\tau \in \cT_\kappa$,
 and we will study the collection of $\ZZ_n \star \ZZ_m$-representations $R(\tau)$. Compatible with the description of the torus
  knot (\ref{eqn:torus}), we will denote by $x, y \in \ZZ_n \star \ZZ_m$ the generators of the $\ZZ_n$ and $\ZZ_m$ parts, respectively.

Following the notation of (\ref{eqn:Wij}), let $m_{i,j}$ be the multiplicity of the isotypic piece $W_{i,j}$ of the
 semi-simple filtration of the type $\tau$. The corresponding eigenvalues of these pieces are denoted 
 $\kappa_{i,j} = (\sigma_{i,j}(x), \sigma_{i,j}(y))$. Then we take
\begin{equation}\label{eq:cI}
	\cI(\tau) = \prod_{i = 1}^s \prod_{j=1}^{s_i} \Sym^{m_{i,j}} \left(R^{\irr}_{\kappa_{i,j}}\right).
\end{equation}

We have a map $\Gr_\bullet: R(\tau) \to \cI(\tau)$ given by $\rho \mapsto \left(\Gr_i(\rho)\right)_i$, where the isotypic
 components of the graded representations are grouped together into the symmetric product. We will call this map the
  \emph{semi-simplification}.

The fiber of the semi-simplification is given by the ways in which this irreducible part can be completed with off-diagonal
 maps. Decompose the graded pieces of the type into its irreducible components as
$$
	\Gr_i(V_\bullet) = U_{\nu_{i-1} + 1} \oplus \ldots \oplus U_{\nu_{i}}\, ,
$$
for an increasing sequence $0 = \nu_0 < \cdots < \nu_{i-1} < \nu_i < \cdots$ This means that the isotypic components 
$W_{i,j}$ will be isomorphic to $m_{i,j}$ of the pieces $U_{\alpha}$, for $\nu_{i-1} < \alpha \leq \nu_i$. Consider $\varrho 
= \bigoplus \varrho_{\alpha} \in \cI(\tau)$, with $\varrho_{\alpha}$ the irreducible action on $U_{\alpha}$. When restricted 
to $\ZZ_m \subset \ZZ_n \star \ZZ_m$, the $\ZZ_m$-module $U_{\alpha}$ is semi-simple with $1$-dimensional irreducible pieces. 
Let $u_{\alpha ,1}, \ldots, u_{\alpha ,d_\alpha}$ be a basis of $U_{\alpha}$, such that $\langle u_{\alpha, l}\rangle$ is an 
irreducible $\ZZ_m$-representation with character $\varepsilon_{\alpha, l}$.

Now, consider the vector space
$$
	\cM_0 = \bigoplus_{\alpha < \beta} \Hom_k(U_{\beta}, U_{\alpha}).
$$
Given $M \in \cM_0$, in the previous basis it can be seen as a collection of maps $M_{\alpha\beta}: U_\beta \to U_\alpha$. Let us consider the induced representation $\varrho_M: \ZZ \star \ZZ \to \GL_r$ by
$$
	\varrho_M(x) (u_{\beta, l}) = \varrho_{\beta}(x) (u_{\beta, l}), \quad
	\varrho_M(y)(u_{\beta, l}) = \varepsilon_{\beta, l}\, u_{\beta, l} + \sum_{\alpha} M_{\alpha\beta}(u_{\beta, l}).
$$
Here $x$ and $y$ denote the generators of each of the factors of $\ZZ \star \ZZ$. Observe that $\varrho_M(x)^n = \bigoplus_\beta \varrho_{\beta}(x)^n = \Id$, but $\varrho_M(y)^m$ may not be the identity map. Hence, $\varrho_M$ can be seen as a representation $\varrho_M: \ZZ_n \star \ZZ \to \GL_r$.

In other words, in a matrix notation in the basis above, the representation is given by a matrix $A$ with diagonal blocks 
$\varrho_\beta(x)$, and by a matrix $B$ whose diagonal blocks are diagonal matrices $B_\beta = \mathrm{diag}(\varepsilon_{\beta, 1},
\ldots, \varepsilon_{\beta, d_\beta})$
and there are upper-diagonal blocks $M_{\alpha \beta}$, $\alpha < \beta$.

To address the problem that $\varrho_M$ is not a representation of $\ZZ_n \star \ZZ_m$, we consider the subset
$$
	{\cM}_1 = \left\{M \in \cM_0 \,|\, \varrho_M(y)^m = \Id \right\}.
$$
This defines a closed subvariety ${\cM}_1 \subset \cM_0$ of admissible off-diagonal maps,
i.e.\ $M \in \cM_0$ such that $\varrho_M$ is a $\ZZ_n \star \ZZ_m$-representation.

This space is actually a linear subspace that can be easily characterized. For $\alpha < \beta$, write into its components
 $$
 M_{\alpha \beta}=\left( m^{ij}_{\alpha \beta}\right)_{i,j}\, .
 $$
We have the following characterization.

\begin{proposition}\label{prop:values-fixed-M}
 If the matrix $M \in \cM_1$ then for $\alpha < \beta$ and $i,j$ such that $\varepsilon_{\alpha, i}=\varepsilon_{\beta, j}$ then 
 $m^{ij}_{\alpha \beta}=c^{ij}_{\alpha \beta}$, a specific value depending on $M$. Otherwise $m^{ij}_{\alpha \beta}$ is free. 
 The converse also holds.  
\end{proposition}

\begin{proof}
 When we take the power $B^m=\Id$, we look at the upper-diagonal block for $\alpha<\beta$. This is equal to
  $$
   \sum_{\alpha \leq \gamma_1\leq \ldots \leq \gamma_{m-1}\leq \beta}  M_{\alpha \gamma_1}M_{\gamma_1\gamma_2}\cdots M_{\gamma_{m-1}\beta} =0,
   $$
where we write $M_{\beta\beta}=B_\beta$. Looking at the terms containing $M_{\alpha\beta}$, we have
 \begin{equation}\label{eqn:values-fixed-M}
 \sum_{t=1}^m B_\alpha^{t-1} M_{\alpha\beta} B_\beta^{m-t} = M',
 \end{equation}
where $M'$ is a fixed matrix depending on the matrices $M_{\g\g'}$ with smaller ``distance''  between $\g,\g'$.
 The coefficients of the left hand matrix are
  $$
  \sum_{t=1}^m m^{ij}_{\alpha\beta} \varepsilon_{\alpha, i}^{t-1} \varepsilon_{\beta, j}^{m-t}= \left\{
  \begin{array}{ll}m^{ij}_{\alpha\beta} \frac{\varepsilon_{\alpha, i}^m- \varepsilon_{\beta, j}^m}{\varepsilon_{\alpha ,i}- \varepsilon_{\beta, j}}=0
   & \text{ if }\varepsilon_{\alpha, i}\neq \varepsilon_{\beta, j} \\[8pt]
  m^{ij}_{\alpha \beta} m \, \varepsilon_{\alpha, i}^{m-1}=  m^{ij}_{\alpha \beta} m\,  \varepsilon_{\alpha, i}^{-1} & \text{ if }\varepsilon_{\alpha, i}
   = \varepsilon_{\beta, j} 
  \end{array}\right.
  $$
  hence the result.
\end{proof}

\begin{remark}\label{rem:values-fixed-M}
When $\nu_{i-1}<\alpha\leq \nu_i$ and $\nu_i<\beta\leq \nu_{i+1}$ (that is, when the modules $U_\alpha$, $U_\beta$
appear in consecutive steps in the semi-simple filtration), we have that the
right hand side of (\ref{eqn:values-fixed-M}) vanishes. Then $m_{\alpha\beta}^{ij}=0$ when $\varepsilon_{\alpha,i}=\varepsilon_{\beta,j}$.
\end{remark}

\begin{lemma} \label{lem:9.9}
Any representation $\rho \in R(\tau)$ is isomorphic to one of the form $\rho = \Gr_\bullet(\rho)_M$ for some $M \in {\cM}_1$.
\end{lemma}

\begin{proof}
Consider the semi-simple filtration for $\rho = (A, B)$ with irreducible pieces $U_\alpha$ as above. 
First, as $\ZZ_n$-representation, $A$ is completely reducible so we can find a basis of $V=k^r$ of $A$-eigenvectors vectors 
$\hat{u}_{\alpha, l}$ such that $\hat{u}_{\alpha, 1}, \ldots, \hat{u}_{\alpha, d_\alpha}$ is a basis of $U_\alpha$. Modify
$\hat{u}_{\alpha, 1}, \ldots, \hat{u}_{\alpha, d_\alpha}$ to another basis $u_{\alpha, 1}, \ldots, u_{\alpha, d_\alpha}$ of $U_\alpha$
 in which $B$ is block-diagonal. In this basis $(A,B) = \Gr_\bullet(\rho)_M$ for $M \in \cM_0$, and the restriction that $B^m = \Id$ 
 forces $M \in \cM_1$.
\end{proof}

This shows that if $\Gr_\bullet: R(\tau) \to \cI(\tau)$ is the semi-simplification map and $\varrho \in \cI(\tau)$ then the fiber 
$\Gr_\bullet^{-1}(\varrho)$ is contained in $\cM_1$. However, this fibre may not be the whole of $\cM_1$, since an element $M \in \cM_1$ may not 
preserve the non-splitting conditions imposed by the semi-simple filtration. For instance, if $M_{\alpha\beta} = 0$ for a fixed 
$\nu_{i} < \beta \leq \nu_{i+1}$ and for all $\nu_{i-1} < \alpha \leq \nu_{i}$ then $i$-th step of the semi-simple filtration is 
not maximal, since it can be extended with $U_\beta$. 
Hence, the characterization of this fiber will require further analysis. 
%that
% will be performed in the following sections.
%
%
%\subsection{Case I. All the isotypic components have multiplicity $1$}

Fix $\tau \in \cT_\kappa$ with $\kappa = (\bm{\epsilon}, \bm{\varepsilon})$.  Consider the `unsymmetrized' versions of
 (\ref{eq:cI}) given by
$$
	\hat{\cI}(\tau) = \prod_{i = 1}^s \prod_{j=1}^{s_i} \left(R^{\irr}_{\kappa_{i,j}}\right)^{m_{i,j}}.
$$
There is a quotient map $\hat{\cI}(\tau) \to \cI(\tau)$. Consider the pullback diagram
\begin{equation}\label{eqn:diagr}
   \xymatrix
   {
   	\hat{R}(\tau) \ar[r] \ar[d]_{\hat{\Gr}_\bullet} & R(\tau) \ar[d]^{\Gr_\bullet} \\
   	\hat{\cI}(\tau) \ar[r] & \cI(\tau).
   }
\end{equation}

Now we shall study $\hat{R}(\tau)$ in detail. For that purpose, let us consider the map $\hat{\cI}(\tau) \to \GL_r$ 
given by $\varrho = (A,B) \mapsto B$. The fiber over $\Sigma_{\bm{\varepsilon}} \in \GL_r$, the diagonal matrix with eigenvalues 
$\bm{\varepsilon}$, will be denoted by 
 \begin{equation}\label{eqn:I0}
 \hat{\cI}_0(\tau)=\big\{\varrho=(A,\Sigma_{\bm{\varepsilon}}) \in \hat{\cI}(\tau) \big\}.
 \end{equation}
These are semi-simple representations whose matrix $B$ is diagonal (not just diagonalizable). 

Fix $\varrho \in \hat{\cI}_0(\tau)$. Let 
 \begin{equation}\label{eqn:cM}
  \cM_\varrho= \big\{ M \in \cM_1\, |\, \varrho_M \in \hat{R}(\tau)\big\},
 \end{equation}
that is, the set of admissible upper-triangular completions of $\varrho$. We also consider the gauge group of $\cM_\varrho$ defined as 
 \begin{equation}\label{eqn:cG}
 \cG_\varrho=\{ P \in \GL_r\, | \, \text{for all }M \in \cM_\varrho, P\varrho_MP^{-1} = \varrho_{M'}, \text{ for some }
 M' \in \cM_\varrho\}.
 \end{equation}

%We say that all isotypic components have multiplicity $1$ if all $m_{i,j}=1$. In this case, 
%we have the following result.

\begin{theorem}\label{thm:11}
%Suppose that $m_{i,j}=1$ for all $i,j$. Then 
$\cM_\varrho$ is an algebraic variety and $\cG_\varrho$ is an algebraic subgroup of $\GL_r$. Moreover, they are independent of the 
representation, in the sense that $\cM_\varrho \cong \cM_{\varrho'}$ and $\cG_\varrho \cong \cG_{\varrho'}$ for any 
$\varrho, \varrho' \in \hat{\cI}_0(\tau)$.
\end{theorem}

\begin{proof}
Once we know that $\cM_\varrho$ is an algebraic variety, then the fact that $\cG_\varrho$ is an algebraic subgroup follows easily. 
Indeed, consider the cartesian square
\[
\begin{displaystyle}
   \xymatrix
   { Z \ar[r]\ar[d] & \cM_{\varrho} \times \GL_r \ar[d]\\
   \cM_{\varrho} \ar[r] & \cM_1\, ,
   }
\end{displaystyle}   
\]
where the bottom morphism is the inclusion and the rightmost map is $(P, \varrho_M) \mapsto \cM_1$. The pullback variety is thus
$$
	Z = \left\{(P, M) \in \GL_r \times \cM_\varrho\,|\, P\varrho_M P^{-1} \in \cM_\varrho \right\}.
$$
By definition, $\cG_\varrho = \pi(Z)$ where $\pi: \GL_r \times \cM_\varrho \to \GL_r$ is the projection onto the first component. 
The projection of an algebraic variety is a constructible set (a disjoint union of locally closed subsets in the Zariski topology).
As $\cG_\varrho$ is a group, it is an homogeneous space. Hence taking a smooth point which has an open neighbourhood where
$\cG_\varrho$ is closed, we see that this happens at every point. Therefore $\cG_\varrho$ is quasi-projective, and hence an algebraic variety.

%an algebraic variety, the result follows.

With respect to $\cM_\varrho$, take $M = \oplus\, M_{\alpha\beta} \in \cM_1$. In order to lie in $\cM_\varrho$, the matrices 
$M_{\alpha\beta}$ must satisfy two conditions. In the first place,
if $\nu_{i-1}<\alpha < \beta \leq \nu_i$ then we must have $M_{\alpha\beta}=0$. Thus we set the linear subspace
 $$
 	\cH = \big\{ M=  \oplus\, M_{\alpha\beta}  \, | \, M_{\alpha\beta}=0 , \text{ if } \nu_{i-1}<\alpha < \beta \leq \nu_i, i=2,\ldots, s \big\} 
 	\subset %\bigoplus_{\alpha < \beta} \Hom\left(U_\beta, U_\alpha\right) \subset 
 	\cM_1.
 $$ 

Second, if $\nu_{i-1}< \alpha \leq \nu_i$ and $\nu_{i} < \beta \leq \nu_{i+1}$, then the maximality of the semi-simple
sub-representations implies  that we cannot have a larger decomposition
 \begin{equation}\label{eqn:decomp}
 \Gr_{i-1}(\varrho_M) \oplus U_{\beta} \subset V/V_{i-1}.
 \end{equation} 
 We have a decomposition (\ref{eqn:decomp}) when there is
 a change of basis such that $M_{\alpha\beta}=0$ for all $\alpha = \nu_{i-1}+1,\ldots, \nu_i$. 
 Let us fix $\alpha, \beta$ and let us only focus on these blocks. There, we have that $\varrho = (A,B)$ and 
 $\varrho_M = (A, B_M)$, where
 $$
 A=\begin{pmatrix} A_\alpha & 0 \\ 0 & A_\beta \end{pmatrix},\quad
 B=\begin{pmatrix} B_\alpha & 0 \\ 0 & B_\beta \end{pmatrix}, \quad
 B_M=\begin{pmatrix} B_\alpha & M_{\alpha\beta} \\ 0 & B_\beta \end{pmatrix},
$$ 
where the matrices $B_\alpha, B_\beta$ are diagonal. Thus a change of basis $P$ has to be of the form
 $$
 P=\begin{pmatrix} P_\alpha & S \\ 0 & P_\beta \end{pmatrix}.
$$
Since $P\varrho_MP^{-1}$ has to lie in $\cM_\varrho$, we must have $P A P^{-1} = A$, and this forces that $P_\alpha=\lambda_\alpha \Id$,
 $P_\beta=\lambda_\beta \Id$ for some $\lambda_\alpha, \lambda_\beta \in k^*$. Now, we consider a change of coordinates of the form
 \begin{equation}\label{eqn:Q}
	Q = \begin{pmatrix} Q_\alpha & 0 \\ 0 & Q_\beta \end{pmatrix},
 \end{equation}
such that
$$
	QAQ^{-1} = \begin{pmatrix} \diag(\epsilon_{\alpha ,1}, \ldots, \epsilon_{\alpha ,d_\alpha}) & 0 \\ 0 & 
	\diag(\epsilon_{\beta, 1}, \ldots,	 \epsilon_{\beta, d_\beta}) \end{pmatrix}.
$$
It is straightforward to check that $S = Q_\alpha \Theta Q_\beta^{-1}$, where $\Theta = (\theta_{ij})$ is a matrix with $\theta_{ij} = 0$ if 
$\epsilon_{\alpha, i} 
\neq \epsilon_{\beta, j}$. In this fashion, the action of $P$ is given by
\begin{align*}
	M_{\alpha\beta}  \mapsto &\, \lambda_\alpha\lambda_\beta^{-1} M_{\alpha\beta}+ \lambda_\beta^{-1} 
	\left(S B_\beta - B_\alpha S\right) \\
	&= \lambda_\alpha\lambda_\beta^{-1} M_{\alpha\beta}+ \lambda_\beta^{-1} \left(Q_\alpha \Theta 
	Q_\beta^{-1} B_\beta - B_\alpha Q_\alpha 
	\Theta Q_\beta^{-1}\right).
\end{align*}
Observe that this action preserves the conditions of Proposition \ref{prop:values-fixed-M}. We consider the linear space 
$\ell_{\alpha\beta} = \langle Q_\alpha \Theta Q_\beta^{-1} B_\beta - B_\alpha Q_\alpha \Theta Q_\beta^{-1} \rangle$. Observe 
that $\ell_{\alpha\beta}$ only depends on the base representation $\varrho$.

In this way, the condition that, after a change of variables, the blocks $M_{\alpha\beta}$ do not vanish for all 
$\nu_{i-1}< \alpha \leq \nu_i$, is that it cannot happen that $M_{\alpha\beta} \in \ell_{\alpha\beta}$ for all $\alpha$. 
 Hence, we take
 $$
 L_\beta =  \bigoplus_{\nu_{i-1} < \alpha \leq \nu_i} \ell_{\alpha\beta} \subset \bigoplus_{\nu_{i-1} < \alpha \leq \nu_i}
  \Hom_k \left(U_\beta,U_\alpha\right) = \Hom_k \left(U_\beta, \Gr_{i-1}(\varrho)\right).
 $$

In the case that all the isotypic components of $\Gr_i(\varrho_M)$ are of multiplicity one (that is, $m_{i,j}=1$ for all $j$),
then the only possible decompositions (\ref{eqn:decomp})
 must involve one of the $U_\beta$. Therefore the condition means that
 $$
 \mathrm{p}_\beta(M)=(M_{\alpha\beta})_{\nu_{i-1} < \alpha \leq \nu_i}  \notin L_\beta ,
 $$
 for all $\nu_{i}< \beta \leq \nu_{i+1}$. So setting
  $$
  \cL =  \big\{ M=  \oplus\, M_{\alpha\beta}  \, | \, \mathrm{p}_\beta(M) \notin L_\beta,
   \text{ for all } \nu_i <\beta \leq \nu_{i+1}, i=1,\ldots, s-1 \big\} ,
 $$
we have that $\cM_\varrho = \cH \cap \cL$. Observe that $\cH$ is independent of $\varrho$ and, 
for any $\varrho, \varrho' \in \hat{\cI}_0(\tau)$, we have that $\cL \cong \cL'$. Hence, we have that $\cM_\varrho \cong \cM_{\varrho'}$.

In the general case, the isotypic components can have multiplicity $m_{i,j}>1$. 
If $U_{\beta_1},\ldots, U_{\beta_k}$ are the isomorphic components, then the condition means that
 $$
 \la\mathrm{p}_{\beta_1}(M),\ldots, \mathrm{p}_{\beta_k}(M)\ra_* \cap L_\beta \neq \emptyset.
 $$
 Here $L_\beta$ is any of the $L_{\beta_j}$, since all $Q_{\beta_j}$ are the same, and $\la v_1, \ldots, v_m\ra_*$ denotes the collection of non-trivial linear combinations of $v_1, \ldots, v_m$, that is of combinations of the form $\lambda_1 v_1 + \ldots + \lambda_m v_m$ with $(\lambda_1, \ldots, \lambda_m) \neq (0, \ldots, 0)$. 
We define an algebraic set
  \begin{equation}\label{eq:cell}
  \begin{aligned}
  \cL =  \big\{ M \, | & \,  \la\mathrm{p}_{\beta_1}(M),\ldots, \mathrm{p}_{\beta_k}(M)\ra_* \cap L_\beta\neq \emptyset, \\
   & \text{ for all } \nu_i <\beta_1,\ldots,\beta_k \leq \nu_{i+1} \text{ isotypic}, i=1,\ldots, s-1 \big\} .
   \end{aligned}
\end{equation} 
The algebraicity of this condition is equivalent to the non-vanishing of the element 
  $$
  \mathrm{p}_{\beta_1}(M)\wedge \ldots\wedge \mathrm{p}_{\beta_k}(M)\in
  \bigwedge\nolimits^k  \left( \Hom\left(U_\beta, \Gr_{i-1}(\varrho)\right)/L_\beta\right).
   $$

As before $\cM_\varrho = \cH \cap \cL$, and $\cM_\varrho \cong \cM_{\varrho'}$ for any $\varrho, \varrho' \in \hat{\cI}_0(\tau)$.
\end{proof}

With this result at hand, we get the following.

\begin{proposition}\label{prop:fibration}
Fix $\kappa$ and a type $\tau \in \cT_\kappa$. The map
$$
	\hat{\Gr}_\bullet: \hat{R}(\tau) \to \hat{\cI}(\tau)
$$
is locally trivial in the Zariski topology. Moreover, the fiber of this map is the quotient $\left(\left(\cM_\varrho \times 
\GL_r\right)/\cG_\varrho\right)/(\prod_{\alpha} \PGL_{d_{\alpha}})$, where $d_\alpha = \dim U_\alpha$.
\end{proposition} 

\begin{proof}
Let us look at the fiber of the map $\hat{\Gr}_\bullet$. 
Fix $\varrho \in \hat{\cI}(\tau)$ and let $\varrho_0 \in \hat{\cI}_0(\tau)$ be the diagonal element conjugated to 
$\varrho$. We have a natural map 
 $$
 \Psi: \cM_{\varrho_0} \times \GL_r  \to \hat{R}(\tau), \quad \Psi(M, Q) = Q(\varrho_0)_MQ^{-1}\, .
 $$ 
Moreover, $\Psi(M,Q) = \Psi(M',Q')$ if and only if 
$Q'^{-1}Q (\varrho_0)_M
 Q^{-1}Q' = (\varrho_0)_{M'}$, which in particular means that $Q'^{-1}Q \in \cG_{\varrho_0}$.

Hence, if we quotient by the natural action of the gauge group we get that the map $\tilde{\Psi}: \left(\cM_{\varrho_0} \times \GL_r\right)/
\cG_{\varrho_0} \to \hat{R}(\tau)$ is injective. Its image is $\hat{\Gr}_\bullet^{-1}(\varrho_0) \times \prod_{\alpha} \PGL_{d_{\alpha}}$,
where the second factor accounts for moving $\varrho_0$ in its isomorphism class (recall that the action of $\PGL_{d_\alpha}$ on each irreducible piece is free). 
Quotienting by this factor the identification of the fiber $\hat{\Gr}_\bullet^{-1}(\varrho_0)$ follows.
The Zariski locally triviality comes 
from the fact that the matrix (\ref{eqn:Q}) can be defined in a Zariski open set, which amounts to diagonalizing
$A$ with a basis in a Zariski neighborhood.
\end{proof}

Due to the previous result, we will loosely write the isomorphism type of $\cM_\varrho$ and $\cG_\varrho$ by $\cM_\tau$ and $\cG_\tau$, respectively. With this notation, we have that

\begin{corollary}\label{cor:formula-X}
For any type $\tau \in \cT_\kappa$ we get that
$$
	[\hat{R}(\tau)] = \frac{[\hat{\cI}(\tau)]}{\prod_{\alpha} [\PGL_{d_{\alpha}}]} \cdot \frac{[\cM_\tau] \cdot [\GL_r]}{[\cG_\tau]}\, .
$$
\end{corollary} 
 
 \begin{proof}
The formula follows directly from Proposition \ref{prop:fibration} taking into account that virtual classes are multiplicative for 
Zariski locally trivial fibrations.
 \end{proof}

Observe that, writing down the multiplicities of the isotypic components in (\ref{eq:cI}) in terms of the pieces $U_\alpha$, we have
$$
	\frac{[\hat{\cI}(\tau)]}{\prod_{\alpha} [\PGL_{d_{\alpha}}]} = \prod_{\alpha} \frac{\left[R^{\irr}_{\kappa_{\alpha}}\right]}{[\PGL_{d_{\alpha}}]} = \prod_{\alpha} \frM_{\kappa_{\alpha}}^{\irr},
$$
where $\kappa_\alpha$ denotes the configuration of eigenvalues of the $\alpha$-block. We will shorten this last factor as $\frM_\tau^{\irr} = \prod_{\alpha} \frM_{\kappa_{\alpha}}^{\irr}$.

\begin{corollary}\label{cor:no-action-discrete}
With the notation of (\ref{eq:cI}), if for every $m_{i,j} > 1$ we have that $\dim W_{i,j} = 1$, i.e.\ if all the repeated irreducible
 representations are $1$-dimensional, then $\cI(\tau) = \hat{\cI}(\tau)$ and the map
$$
	R(\tau) \to \cI(\tau)
$$
is locally trivial in the Zariski topology with fiber $\left(\left(\cM_\varrho \times \GL_r\right)/\cG_\varrho\right)/(\prod_\alpha \PGL_{d_\alpha})$.
\end{corollary} 

\begin{proof}
We know that $R^{\irr}_{\kappa}$ is a single point for $1$-dimensional representations, so in this case the action of the 
symmetric group $S_{m_{i,j}}$ on
$(R^{\irr}_{\kappa})^{m_{i,j}}$ is trivial. Thus, $\hat{\cI}(\tau) = \cI(\tau)$ and the result follows from 
Proposition \ref{prop:fibration}.
\end{proof}

Therefore, under the hypothesis of Corollary \ref{cor:no-action-discrete}, we have the analogous formula to Corollary \ref{cor:formula-X}
 \begin{equation}\label{eqn:working-eqn3}
	[R(\tau)] = [\frM_\tau^{\irr}] \cdot \frac{[\cM_\tau] \cdot [\GL_r]}{[\cG_\tau]}.
 \end{equation}

In particular, for rank $r\leq 4$, the condition of Corollary \ref{cor:no-action-discrete} holds. The only exception
is $r=4$, $s=1$, $m_{1,1}=2$, $\dim W_{1,1}=2$, but in this case all representations are reducible (see Section \ref{subsec:8.2}), 
so it does not contribute to $R_r^{\irr}$.

%%%%%%%%%%%%%%%%%%%%%%%%%%%%%%%%%%
\section{Explicit formulas for some strata} \label{sec:explicit-formulas}
%%%%%%%%%%%%%%%%%%%%%%%%%%%%%%%%%%

To compute the motive of $[\cR^{\irr}(\varpi)]$ we need to join the equations (\ref{eqn:working-eqn1}), (\ref{eqn:working-eqn2})
and (\ref{eqn:working-eqn3}), where $\cI_0(\tau)$, $\cM_\tau$ and $\cG_\tau$ are given in (\ref{eqn:I0}), (\ref{eqn:cM})
and (\ref{eqn:cG}).

The proof of Theorem \ref{thm:11} is constructive and gives explicit descriptions of $\cM_\tau$ and $\cG_\tau$. 
Fix two values $\alpha, \beta$ in different steps of the semi-simple filtration. % with $\nu_{i-1}< \alpha \leq \nu_i < \beta \leq \nu_{i+1}$.
Let $\epsilon_1, \ldots, \epsilon_a$ (resp.\ $\varepsilon_1, \ldots, \varepsilon_b$) be the eigenvalues of $A_\alpha$ and $A_\beta$ 
(resp.\ $B_\alpha$ and $B_\beta$) and let $a^\alpha_{k}$ and $a^\beta_{k}$ (resp.\ $b^\alpha_{ k}$ and $b^\beta_{k}$) be the 
multiplicities of $\epsilon_k$ in $A_\alpha$ and $A_\beta$ (resp.\ of $\varepsilon_k$ in $B_\alpha$ and $B_\beta$). Then, 
$b^\alpha_kb^\beta_k$ is the number of times that $\varepsilon_k$ is a repeated eigenvalue in $B_\alpha$ and $B_\beta$ and thus, 
according to Proposition \ref{prop:values-fixed-M}, the elements of $\cM_1$ have $\sum_k b_{k}^\alpha b_{k}^\beta$ fixed elements. 
So the contribution to the dimension of $\cM_1$ of the pair $\alpha, \beta$ is 
 $$
 C(\alpha,\beta)=\dim U_\alpha \cdot \dim U_\beta - \sum_k b_{k}^\alpha b_{k}^\beta\, .
  $$
  
Define $d(\alpha) = i$, if $U_\alpha$ belongs to the $i$-th step of the semi-simple filtration i.e.\ 
$\nu_{i-1} < \alpha \leq \nu_i$. In the same vein, define $c_i(\beta) = j$, if $U_\beta$ belongs to 
the $i$-th step of the semi-simple filtration and it is a component of the $j$-th isotypic factor $W_{i,j}$, that
is $U_\beta \cong W_{i,j}$.
%For this purpose, fix $i$ and let us consider the componets $U_{\beta}$ for $\nu_i < \beta \leq \nu_{i+1}$ of the $(i+1)$-th step of the semi-simple filtration. Group together the components  so that we find a sequence $\nu_i = \mu^{i+1}_0 < \mu^{i+1}_1 < \ldots < \mu^{i+1}_{s_{i}} = \nu_{i+1}$ such that the blocks  $U_{\beta}$ for $\mu^{i+1}_{j-1} < \beta \leq \mu^{i+1}_j$ are the isomorphic blocks corresponding to the isotypic component $W_{i+1,j}$.
In this way, the contribution to $\cM_1$ of the blocks corresponding to the $j$-th  isotypic component is a vector space of dimension
 \begin{align*}
 C_{i,j}&=\sum_{d(\alpha) = i}\, \, \sum_{c_{i+1}(\beta) = j} C(\alpha,\beta).
 \end{align*}

To this contribution, we have to subtract the elements in $\cM_1 - \cM_\tau$ corresponding to the forbidden configurations. 
These are given by the ``Schubert cells'' (\ref{eq:cell}),
%\begin{align*}
%	\cL_{i,j} = \prod_{\nu_{i-1} < \alpha \leq \nu_i} \big\{ \la M_{\alpha,\mu^{i+1}_{j-1} +1 },\ldots, M_{\alpha, \mu^{i+1}_{j}}\ra
%	 \cap \ell_{\alpha\mu^{i+1}_{j}} \neq 0\big\} .
%\end{align*}
\begin{align*}
\cD_{i,j} = \big\{\la\mathrm{p}_{\beta_1}(M),\ldots, \mathrm{p}_{\beta_k}(M)\ra_* \cap L_\beta \neq \emptyset 
\,|\, c_{i+1}(\beta_1) = \ldots = c_{i+1}(\beta_k) = j \big\}.
\end{align*}
Hence, the contribution of the blocks in consecutive steps is
$$
	\prod_{i=1}^s \prod_{j=1}^{s_i} \left(k^{C_{i,j}} - \cD_{i,j}\right).
$$

On the other hand, the contribution of the non-consecutive blocks $\alpha, \beta$ is just given by a vector space of dimension 
$$
	C = \sum_{d(\beta) - d(\alpha) > 1} C(\alpha,\beta).
$$
Therefore, the motive of $\cM_\tau$ is given by
 \begin{equation}\label{eqn:cMtau}
 [\cM_\tau] = [k^C] \cdot \prod_{i=1}^s \prod_{j=1}^{s_i} \left[k^{C_{i,j}} - \cD_{i,j}\right] = q^C \prod_{i=1}^s 
 \prod_{j=1}^{s_i} \left(q^{C_{i,j}} - [\cD_{i,j}]\right).
 \end{equation}
 
Now, let us give formulas for $[\cD_{i,j}]$ in some particular cases.
\begin{enumerate}
	\item\label{item:mult-1} If $m_{i+1,j} = 1$, that is the isotypic component
$W_{i+1,j}$ has of multiplicity $1$, observe that the Schubert cell is just
\begin{align*}
	\cD_{i,j} = L_\beta = \bigoplus_{d(\alpha) = i} \ell_{\alpha\beta} = \bigoplus_{d(\alpha) = i}\langle
	 Q_\alpha \Theta Q_\beta^{-1}B_\beta -B_\alpha Q_\alpha \Theta Q_\beta^{-1}\rangle,
\end{align*}
where $\beta$ is the only index with $c_{i+1}(\beta) = j$. 
The matrix $\Theta$ moves in a vector space of dimension $\sum_k a_{k}^\alpha a_{k}^\beta$. 
The kernel of the map $\Theta \mapsto  Q_\alpha \Theta Q_\beta^{-1}B_\beta -B_\alpha Q_\alpha \Theta Q_\beta^{-1}$
is given by the homomorphisms as representations, $\Hom_\G(U_\beta,U_\alpha)$,
since 
%Observe that we have a trivial estimate $0 \leq \dim \ell_{\alpha\beta} \leq \sum_k a_{k}^\alpha a_{k}^\beta$ and, indeed, 
%$\dim \ell_{\alpha\beta} = 0$ if and only if $U_\alpha$ and $U_\beta$ are isomorphic since, in that case 
$S = Q_\alpha \Theta Q_\beta^{-1}$ is a homomorphism of representations
if and only if $SB_\beta=B_\alpha S$ (as $S A_\beta=A_\alpha S$ holds by the condition on $\Theta$). 
Adding up over $\alpha$, we have that $\cD_{i,j}$ is subspace of dimension
 \begin{align*}
 D_{i,j}&=\dim L_\beta= \sum_{d(\alpha) = i} \dim \ell_{\alpha\beta} = 
  \sum_k a_k^\alpha a_k^\beta-\dim \Hom_\G(U_\beta, \Gr_i(V_\bullet)).
 \end{align*}
Therefore, the formula (\ref{eqn:cMtau}) can be written as %stands for
\begin{equation}\label{eq:mult-1}
 [\cM_\tau] = [k^C] \cdot \prod_{m_{i,j}> 1}  \left(q^{C_{i,j}} - [\cD_{i,j}]\right)
  \cdot \prod_{m_{i,j} = 1} \left(q^{C_{i,j}} - q^{D_{i,j}}\right).
\end{equation}
%	\item Suppose that $\dim U_\alpha = \dim U_\beta = 1$ and $U_\beta$ has multiplicity $1$. 
%If $B_\alpha=B_\beta$, then $C(\alpha, \beta) = 0$ so the $\alpha, \beta$ block makes no contribution. 
%On the other hand, if $A_\alpha=A_\beta$ but $B_\alpha \neq B_\beta$ we have that $C(\alpha, \beta) = 1$ 
%and $\dim \ell_{\alpha\beta} = 1$. In this case, as we will see, the gauge group $\cG_\tau$ has an extra $q$ 
%so, finally, this block also makes no contribution either. Observe that, in particular, if $r = 2$ this implies 
%that $\cM_{\tau} = 0$ (i.e.\ all the representations split).
	\item Suppose that $\dim W_{i+1, j} = 1$ and $m_{i+1,j} = 2$. Denote by $\beta_1, \beta_2$ the 
	indices in the $(i+1)$-th step with $c_{i+1}(\beta_1) = c_{i+1}(\beta_2) = j$. Suppose also that 
	$s_i=1$, with $\alpha = \nu_i$ the only index in the $i$-th step. In that case, 
	$\mathrm{p}_{\beta_1}(M), \mathrm{p}_{\beta_2}(M) \in k$ so always 
	$0 \in \la\mathrm{p}_{\beta_1}(M), \mathrm{p}_{\beta_2}(M) \ra_* \cap L_{\beta}$. 
	In other works, by a change of basis, on the vector $(\mathrm{p}_{\beta_1}(M), \mathrm{p}_{\beta_2}(M))$ 
	we can always arrange that a component is zero. 
	This implies that $D_{i,j} = C_{i,j}$, and hence $[\cM_\tau]=0$. So this case cannot happen.

\item\label{item:mult-2} Suppose that $\dim W_{i+1, j} = 1$ and $m_{i+1,j} = 2$ but now $s_{i} = 2$, 
with two blocks $U_{\alpha_1}, U_{\alpha_2}$ of multiplicity $1$. Denote the eigenvalues of 
$A_{\alpha_1}, A_{\alpha_2}$ (resp.\ $B_{\alpha_1}$ and $B_{\alpha_2}$) by $\epsilon_{\alpha,1},\epsilon_{\alpha,2}$ 
(resp.\ $\varepsilon_{\alpha,1},\varepsilon_{\alpha,2}$) and let the eigenvector of $A_{\beta_1}=A_{\beta_2}$
(resp.\ $B_{\beta_1}= B_{\beta_2}$) be $\epsilon_\beta$ (resp.\ $\varepsilon_\beta$). Suppose that 
$\epsilon_{\alpha,1},\epsilon_{\alpha,2} \neq \epsilon_{\beta}$ and $\varepsilon_{\alpha,1},\varepsilon_{\alpha,2} \neq \varepsilon_{\beta}$.
In that case, we have that $\Theta = 0$ and thus $\ell_{\alpha_1\beta} = \ell_{\alpha_2\beta} = 0$. 
This implies that $L_\beta = 0$ so the Schubert cell condition is that $0 \in \la\mathrm{p}_{\beta_1}(M), \mathrm{p}_{\beta_2}(M) \ra_*$,
where $\mathrm{p}_{\beta_1}(M), \mathrm{p}_{\beta_2}(M)$ are two $2$-dimensional vectors. 
This happens if and only if $\mathrm{p}_{\beta_1}(M), \mathrm{p}_{\beta_2}(M)$ are linearly dependent, 
so the corresponding Schubert cell are the matrices of rank $\leq 1$. That is, $[\cD_{i,j}] = q^4 - [\GL_2] = q^3 + q^2 - q$.
\end{enumerate}

\begin{remark}
These are the only cases that we will find when computing irreducible representations for $r = 4$. 

Observe that the cases $\epsilon_{\alpha,1} = \epsilon_{\beta}$ or $\varepsilon_{\alpha,1} = \varepsilon_{\beta}$ 
in item (\ref{item:mult-2}) do not need to be considered since in that case there is triple eigenvector in either $A$ 
or $B$ and a double eigenvector in the other matrix, so all the representations are reducible (c.f.\ Remark \ref{rmk:reducible}).
\end{remark}

A formula for $\cG_\tau$ is easier to obtain. Recall the definition (\ref{eqn:cG}). Any $P\in \cG_\varrho$
induces an automorphism of $\Gr_i(\varrho)$, respecting the components. By Schur lemma, it acts on the
isotypic components, so producing $\prod_{i=1}^s\prod_{j=1}^{s_i} \GL_{m_{i,j}}$. This has
to be completed with upper diagonal blocks. The matrix $P$ must fix the diagonal matrix $A$, so the upper diagonal block of $P$ corresponding
to $\alpha, \beta$ has to be $S = Q_\alpha \Theta Q_\beta^{-1}$, as in Theorem \ref{thm:11}. The matrix $\Theta$ has
$\sum_k a_k^\alpha a_k^\beta$ free entries, so adding up on the possible matrices, the upper diagonal parts form a vector space of dimension
 $$
 D= \sum_{d(\beta)>d(\alpha)} \sum_k a_k^\alpha a_k^\beta\, .
 $$
Hence
\begin{equation}\label{eq:gauge-group}
  [\cG_\tau]= q^D \prod_{i=1}^s\prod_{j=1}^{s_i} [\GL_{m_{i,j}}].
\end{equation}

\subsection*{Multiplicity}

Recall that, given a multiset $X = \left\{x_1^{m_1}, \ldots, x_n^{m_n}\right\}$, we can consider the subgroup $S_X \subset S_n$ of 
permutations $\pi$ that preserve the multiplicity of the multiset, that is $m_{\pi(i)} = m_{i}$ for all $1 \leq i \leq n$.

Now, consider a configuration of eigenvalues $\kappa = (\bm{\epsilon}, \bm{\varepsilon})$ and let $S_\kappa = S_{\bm{\epsilon}} \times 
S_{\bm{\varepsilon}}$. Given a type $\tau = (\xi, \sigma) \in \cT_\kappa$ and $\pi \in S_\kappa$ we set
$\pi \cdot \tau = (\xi, \pi \cdot \sigma)$. That is, the new type has the same shape $\xi$ but the new eigenvalues, 
$\pi \cdot \sigma$, are the ones of $\sigma$ permuted according to the permutation $\pi$. This induces an action of 
$S_{\kappa}$ on $\cT_\kappa$. 

Moreover, by the computations above, we have that $[R(\tau)] = [R(\pi \cdot \tau)]$. This implies that, if we denote the length of 
the orbit of $\tau$ by $m_\kappa(\tau)$, called the multiplicity of the type $\tau$, we can write
$$
	\sum_{\tau \in \cT_\kappa} [R(\tau)] = \sum_{[\tau] \in \cT_\kappa/S_\kappa} m_\kappa(\tau)[R(\tau)].
$$
This formula reduces drastically the number of types that need to be considered and it is the version that we will use in 
Sections \ref{sec:rank2} and \ref{sec:rank4}.

%%%%%%%%%%%%%%%%%%%%%%%%%%%%%%%%%%%%%%%%%%%%%%
\section{Count of components}\label{sec:count-comp}
%%%%%%%%%%%%%%%%%%%%%%%%%%%%%%%%%%%%%%%%%%%%%%

In this section, we are going to count the valid configurations of eigenvalues $\kappa = (\kappa_1, \kappa_2)$. As we will see, we need 
to consider different cases according to the number of repeated eigenvalues in $\kappa_1$ and $\kappa_2$. We will analyze 
completely the patterns that appear for $r = 4$, and the general case when $r, n, m$ are coprime is studied in Section \ref{subsec:coprime}.

%From now on wedenote the abelian group of $s$-th roots of unit by 
% $$
% \mu_s = \left\{z \in \CC\,|\, z^s=1\right\}.
% $$ 
We fix $n,m$ and the torus knot group (\ref{eqn:torus}). Fix also $r>0$ and recall that we are interested in the variety 
(\ref{eqn:cXr}), and particularly, we want to compute $\cR_r^{\irr}$ using (\ref{eqn:working-eqn1}). By Lemma \ref{lem:red}
these are divided according to a root of unity $\varpi\in \mu_r$. 
If $\epsilon_1,\ldots, \epsilon_r$ are the eigenvalues of $A$, then $\epsilon_i^n=\varpi$ for all $i$, and $\epsilon_1\cdots 
\epsilon_r =1$. To count the distribution of eigenvalues, we introduce the set
%To account for the condition on the determinant, recall
%that $\cR^{\irr}(\varpi) \cong (R_r^{\xi_1,\xi_2})^{\irr}$ by (\ref{eqn:eta-eta}), where $\xi_1=\eta_1^r\in \mu_n$
%and $\xi_2=\eta_2^r\in \mu_m$. We denote, for $\xi \in \mu_n$,
 \begin{equation}\label{eqn:Nn}
	\hat N_{n,r}(\varpi) = \big\{(\epsilon_1, \ldots, \epsilon_r) \in \mu_{nr}^r \,|\, \epsilon_i^n=\varpi, \epsilon_1\cdots 
	\epsilon_r = 1 \big\}.
 \end{equation}
There is an action of the symmetric group $S_r$ on $\hat N_{n,r}(\varpi)$ by permutation of the elements 
$\epsilon_i$, and we denote by $N_{n,r}(\varpi, \xi) = \hat N_{n,r}(\varpi)/S_r$ the corresponding unordered set. Therefore, 
the possible configurations of eigenvalues (\ref{eqn:kappa}) are pairs
$$
	\kappa = (\kappa_1, \kappa_2) \in M_{n, m,r} = \bigsqcup_{\varpi \in \mu_r} N_{n,r}(\varpi) \times N_{m,r}(\varpi).
$$

Moreover, consider a partition $\pi = \left\{1^{e_1}, 2^{e_2}, \ldots, r^{e_r}\right\}$ of $r$ with $e_i \geq 0$ for $1 \leq i \leq r$, 
and $r = \sum_i i e_i$. Let us denote $S_\pi = (S_{1})^{e_1} \times (S_{2})^{e_2} \times \ldots \times (S_{r})^{e_r}$ as a subgroup of 
$S_r$. We will denote by $\hat N_{n,r}^\pi(\varpi)$ the set of points of $\hat N_{n,r}(\varpi)$ whose stabilizer under the action 
of $S_r$ is (conjugated to) $S_\pi$. In other words, $\hat N_{n,r}^\pi(\varpi)$ is the set of tuples of $\hat N_{n,r}(\varpi)$ with $e_i$ 
collections of $i$ equal elements, for $1 \leq i \leq r$. Denote also by $N_{n,r}^\pi(\varpi) = \hat N_{n,r}^\pi(\varpi)/S_r$ 
the collection of unordered tuples and set
 \begin{equation}\label{eqn:cxc}
 M_{n,m,r}^{\pi_1, \pi_2} = \bigsqcup_{\varpi \in \mu_r} N_{n,r}^{\pi_1}(\varpi) \times N_{m,r}^{\pi_2}(\varpi).
 \end{equation}

For given partitions $\pi_1, \pi_2$, 
if $\kappa, \kappa' \in M_{n,m,r}^{\pi_1, \pi_2}$ then we have $[\cR_{r, 
\kappa}^{\red}] = [\cR_{r, \kappa'}^{\red}]$. This follows from (\ref{eq:dec-irred-kappa}) thanks
to the formula in Corollary \ref{cor:formula-X}, which shows that the motives only depend
on how many eigenvalues are repeated and not their particular values. The formula only
gives the motives under the assumptions of Corollary \ref{cor:no-action-discrete}, which
is all that we need. However, working with the diagram (\ref{eqn:diagr}), we can push-down to $\cI(\tau)$
and show that the classes are equal even if we do not have explicit formulas.

The above means that it is enough to compute $[\cR_{r, \kappa}^{\red}]$ for one $\kappa \in 
M_{n,m,r}^{\pi_1,\pi_2}$ (which is a motive independent of $n$ and $m$) 
and then multiply the result by the number of components 
$|M_{n,m,r}^{\pi_1, \pi_2}|$ (a combinatorial number dependent only on $n,m,r$). That is, (\ref{eqn:working-eqn1})
and (\ref{eqn:yyy}) read as
 $$
 [\cR_r^{\irr}]=\sum_{\pi_1,\pi_2} c(\pi_1,\pi_2) [R_{\kappa(\pi_1,\pi_2)}^{\irr}],
 $$
where $c(\pi_1,\pi_2) =|M_{n,m,r}^{\pi_1, \pi_2}|$, and $\kappa(\pi_1,\pi_2)$ is one element in $M_{n,m,r}^{\pi_1, \pi_2}$.

Now we compute the value of $c(\pi_1,\pi_2)$ for some classes of partitions, which will cover at least all situation
for $r=4$. We take $k=\CC$ in the computations below, but the result holds as well for an algebraically closed field of
zero characteristic, as it only involves roots of unity.

\subsection{Partition $ \pi = \left\{1^r\right\}$}
We fix $\alpha_1,\ldots,\alpha_n$ the $n$-th roots of unity, $\alpha_j=e^{2\pi i j /n}$. Since $\varpi\in \mu_r$,
we can write $\varpi=\psi^a$ for some $a=0,\ldots, r-1$, where $\psi=e^{2\pi i/r}$. 
Recalling (\ref{eqn:Nn}), we take
$\zeta=e^{2\pi i /rn}$ so that
 \begin{equation*}
  \epsilon_j = \zeta^a \alpha_{i_j} , \quad 1\leq i_1,\ldots,i_r\leq n, i_j\neq i_k \text{ for } j\neq k.
  \end{equation*}
We introduce the multi-index $I=\{i_1,\ldots, i_r\}$ and consider the sum
 $$
 \prod_{j=1}^n (1+ \zeta^a\alpha_j t) = \sum_{r\geq 0} \sum_{|I|=r} \zeta^{ar}  \alpha_{i_1}\cdots \alpha_{i_r} t^r .
 $$
We want to count those products $\zeta^{ar}\alpha_{i_1}\cdots \alpha_{i_r}=1$. For this we make use of the ``projection'' operator
 $$
  \sum_{k=0}^{n-1} \alpha^k =\left\{ \begin{array}{ll} n , & \text{if } \alpha=1, \\ 0, & \text{otherwise}, \end{array}\right.
  $$
for $\alpha\in \mu_n$. Hence
 $$
 \sum_{k=0}^{n-1} \prod_{j=1}^n (1+ \zeta^{ak}\alpha_j^k t) = 
 \sum_{r\geq 0} \sum_{|I|=r} \sum_{k=0}^{n-1} 
 \zeta^{akr} ( \alpha_{i_1}\cdots \alpha_{i_r})^k t^r = \sum_{r\geq 0} n\, |N_{n,r}^\pi (\omega)|\, t^r.
 $$
Now we redo the sum as
 $$
  \sum_{r\geq 0} n\, |N_{n,r}^\tau(\omega)| t^r = \sum_{d|n} 
  \sum_{s\in \ZZ^*_{e}} \prod_{j=1}^n \left(1+ (\zeta^{ad}\alpha_j^d)^s 
  t\right) ,
 $$
 where $d|n$, $e=n/d$ and $\ZZ_e^*$ are the units of $\ZZ_e$, interpreting that $\ZZ_1^*=\{0\}$. The elements $\alpha_j^d$ 
 are the $e$-th
 roots of unity, repeated $d$ times. We subdivide them into groups.
 Let us consider the first group $\alpha_1^d,\ldots, \alpha_e^d$. The polynomial 
 \begin{equation}\label{eqn:lll}
 \prod_{j=1}^{e} \left(1+ (\zeta^{ad}\alpha_j^d)^s t\right) =(1+(-1)^{e+1} \psi^{as} t^e), %^{d},
 \end{equation}
 since the $e$ complex numbers $-(\zeta^{ad}\alpha_1^d)^s, \ldots,- (\zeta^{ad}\alpha_n^d)^s$ are exactly those 
$z$ such that $ z^e=(-1)^e\zeta^{ad s e}= (-1)^e \zeta^{asn}=(-1)^e\psi^{as}$. The other groups give the
same answer, so we have to take the $d$-th power of (\ref{eqn:lll}).
Putting all together,
 $$
  \sum_{r\geq 0} n\, |N_{n,r}^\pi (\varpi)| t^r = \sum_{e|n} \sum_{s\in \ZZ^*_{e}} (1+ (-1)^{e+1} \psi^{as} t^e)^{n/e},
 $$
 and taking the term $t^r$,
    \begin{equation}\label{eqn:Npi}
       |N_{n,r}^\pi(\varpi)| =\frac1n \sum_{e|n, e|r} \sum_{s\in \ZZ^*_{e}} (-1)^{r+r/e} \binom{n/e}{r/e}  \psi^{as r/e}\, .
  \end{equation}
Here $\psi^{r/e}$ is a primitive $e$-root of unity, and $\psi^{sr/e}$ runs over all primitive $e$-th roots of unity
$\mu_e^*$. The final formula is:
    \begin{equation}\label{eqn:Npi2}
       |N_{n,r}^\pi (\varpi)| = \sum_{e|n, e|r} \frac{ (-1)^{r+r/e} }{n} \binom{n/e}{r/e}  \sum_{\nu\in \mu^*_e} \nu^{a}\, .
  \end{equation}

\begin{remark}\label{rem:Mob1}
The last term is can be computed via the M\"obius inversion formula: %(see \texttt{en.wikipedia.org/wiki/Root\_of\_unity})
 $$
  \sum_{\nu\in \mu^*_e} \nu^{a} = \sum_{y|e} \mu(y)\, \textrm{SR}(e/y;a)
  $$
where $\mu(y)$ is the M\"obius function, and $\textrm{SR}(e;a)=\sum_{\nu\in \mu_e} \nu^a$. So 
$\textrm{SR}(e;a)= e$ when $e|a$ and $0$ otherwise.
\end{remark}
 
 \subsection{Partition $\pi' = \left\{1^{r-2},2^1\right\}$}
We start with the the sum
 $$
  \sum_{p=1}^n \zeta^{2a}\alpha_p^2t^2 \prod_{j\neq p} (1+ \zeta^a\alpha_j t) = 
   \sum_{r\geq 0} \sum_{|I|=r-2 \atop
   p\notin I} \zeta^{ar}  \alpha_p^2\alpha_{i_3}\ldots \alpha_{i_r} t^r\, .
 $$
Taking the projection operator as before,
 $$
 \sum_{k=0}^{n-1} \sum_{p=1}^n \zeta^{2ak}\alpha_p^{2k} t^2 \prod_{j\neq p} (1+ \zeta^{ak}\alpha_j^k t) 
 = \sum_{r\geq 0} n\, |N_{n,r}^{\pi'}(\varpi)|\, t^r.
 $$
Now we redo the sum, using that $(1+\zeta^{ak}\alpha_p^kt)^{-1}=\sum\limits_{l\geq 0} (-1)^l\zeta^{lak}\alpha_p^{lk}t^l$,
\begin{equation}\label{eqn:ttt} 
\begin{aligned}
  \sum_{r\geq 0} n\, | N_{n,r}^{\pi'} &(\varpi)|\, t^r =  \sum_{k=0}^{n-1} \sum_{p=1}^n  \sum_{l=2}^\infty (-1)^l \zeta^{lak}
  \alpha_p^{lk} t^l
  \prod_{j=1}^n (1+ \zeta^{ak}\alpha_j^k t)  \\
  &= \sum_{p=1}^n \sum_{l=2}^\infty \sum_{d|n} \sum_{s\in \ZZ^*_{e}} (-1)^l \zeta^{lads}\alpha_p^{lds} t^l 
      \prod_{j=1}^n (1+ (\zeta^{ad}\alpha_j^d)^s t) \\
  &=  \sum_{l=2}^\infty \sum_{e|n} \sum_{s\in \ZZ^*_{e}} (-1)^l \zeta^{lads} \left(\sum_{p=1}^n\alpha_p^{lds}\right) t^l
    (1+ (-1)^{e+1}\psi^{as} t^e)^{n/e} \\
  &=  n \sum_{e|n} \sum_{s\in \ZZ^*_{e}}   \sum_{l\geq 2 \atop n|lds}^\infty  (-1)^l \zeta^{lads}  t^l
    (1+ (-1)^{e+1} \psi^{as} t^e)^{n/e} .
 \end{aligned}
 \end{equation}
 Here $n|lds$ is equivalent to $e|l$. The term with $e=1$ is 
  $$
  n\sum_{l\geq2} (-1)^l \zeta^{l ans} t^l     (1+ \psi^{as} t)^{n} = n \psi^{2as} t^2   (1+ \psi^{as} t)^{n-1}.
  $$
The term with $e\geq 2$ reduces to
  \begin{align*}
  n \sum_{e|n} \sum_{s\in \ZZ^*_{e}}  &  \sum_{e|l}  (-1)^l \zeta^{lads}  t^l (1+ (-1)^{e+1}\psi^{as} t^e)^{n/e}  \\
 &=n \sum_{e|n} \sum_{s\in \ZZ^*_{e}}  (-1)^e \psi^{as} t^e  (1+ (-1)^{e+1}\psi^{as} t^e)^{n/e-1}.
  \end{align*}
Taking the term $t^r$ in (\ref{eqn:ttt}), and recalling that $\psi^r=1$, we have
  \begin{equation}\label{eqn:Npi'}
  |N_{n,r}^{\pi'}(\varpi)| = \binom{n-1}{r-2} -
      \sum_{e|n, e|r \atop e\geq 2} \sum_{s\in \ZZ^*_{e}}  (-1)^{r +r/e} \binom{n/e-1}{r/e-1}  \psi^{as r/e}.
   \end{equation}

\begin{remark}\label{rem:6.2}
There is an alternative way of proving the expression (\ref{eqn:Npi'}). 
Write $L = |N_{n,r}^{\pi}(\varpi)|$ and $L' = |N_{n,r}^{\pi'}(\varpi)|$. We start with the expression
 $$ 
 \sum_{p=1}^n (1-\zeta^{2a} \alpha_p^2 t^2) \prod_{i\neq p} (1 + \zeta^a \alpha_i t).
 $$
The terms with $\alpha_p^2$ contribute $ - L'\,t^r$, but there are terms
$\sum_{p=1}^n \prod_{i\neq p} (1 + \zeta^a \alpha_i t)$ which contribute with $(n-r) L\, t^r$.
The coefficient $n-r$ comes from the fact that $\alpha_{i_1}\cdots \alpha_{i_r}$ appears in
the $n-r$ summands for $p\in\{1,\ldots,n\}-\{i_1,\ldots, i_r\}$.
Now we project as we have done before, so that
$n(-L'  + (n-r) L)$ equals the $t^r$-coefficient of
 $$
 \sum_{k=1}^n \sum_{p=1}^n (1-\zeta^{2ak} \alpha_p^{2k} t^2) \prod_{i\neq p} (1 + \zeta^{ak} \alpha_i^k t)
 = \sum_{k=1}^n \sum_{p=1}^n (1 -\zeta^{ak} \alpha_p^{k} t) \prod_{i=1}^n (1+ \zeta^{ak} \alpha_i^k t).
 $$
The first summand gives the contribution
$n^2 \sum_{e|n} \sum_{s\in\ZZ_e^*} (1+ (-1)^e \psi^{as}t^{e+1})^{n/e}$. The coefficient of $t^r$ is again
$n^2 L$.
The second summand only contributes for $k=n$ (that is, when $e=1$), and so produces
$-\psi^{a} n t (1+\psi^{a} t)^{n}$.
Altogether this gives the equality
 $$
 -nL'  + n(n-r) L = n^2 L - n \psi^{ar} \binom{n}{r-1},
 $$
which is rewritten, using that $\psi^r=1$, as
 \begin{equation}\label{eqn:LLL}
 L'= - rL + \binom{n}{r-1}.
 \end{equation}
Finally we use the expression of $L$ in (\ref{eqn:Npi}), 
 $$
L = \sum_{e|n, e|r} \frac{(-1)^{r+r/e}}{n}  \binom{n/e}{r/e} \sum_{s\in\ZZ_e^*} \psi^{as r/e}\, .
 $$
The coefficient corresponding to $e=1$ in (\ref{eqn:LLL}) is 
$$
 - \frac{r}{n} \binom{n}{r} + \binom{n}{r-1} = \binom{n-1}{r-2}.    
$$
Hence 
 $$
L'=  \binom{n-1}{r-2} - 
  \sum_{e|n,  e|r \atop e\geq 2} (-1)^{r+r/e} \binom{n/e-1}{r/e-1} \sum_{s\in\ZZ_e^*} \psi^{as r/e},
 $$
using that $\binom{n/e}{r/e}=\frac{n}{r} \binom{n/e-1}{r/e-1}$. 
\end{remark}

 \subsection{Partition $\pi'' = \left\{1^{r-3},3^1\right\}$}
Working as in Remark \ref{rem:6.2} for  $L'' = |N_{n,r}^{\pi''}(\varpi)|$, and starting with 
 $$ 
 \sum_{p=1}^n (1+\zeta^{3a} \alpha_p^3 t^3) \prod_{i\neq p} (1+\zeta^a \alpha_i t),
 $$
we reach the equation $L''+(n-r) L = n L - \binom{n}{r-1} +$ the $t^r$-coefficient of the expression 
$\frac1n \sum\limits_{k=1}^n\sum\limits_{p=1}^n \zeta^{2ak}\alpha_p^{2k}t^2 \prod\limits_{i=1}^n (1+\zeta^{ak}\alpha_i^kt)$.
This coefficient is 
 $$
 \binom{n}{r-2} +  \sum_{2|n,2|r} (-1)^{r/2-1} \binom{n/2}{r/2-1}\psi^{ar/2} ,
 $$ 
where the sum means that there is only one term that appears when $2|n$ and $2|r$.
Substituting the value of $L$ and simplifying, we get 
\begin{equation}\label{eqn:pi''}
 \begin{aligned}
 |N_{n,r}^{\pi''}(\varpi)|= L'' =&\, \binom{n-1}{r-3} + 
  \sum_{e|n,  e|r \atop e\geq 3} (-1)^{r+r/e} \binom{n/e-1}{r/e-1} \sum_{s\in\ZZ_e^*} \psi^{as r/e} \\
  &  -   \sum_{2|n,  2|r } (-1)^{r/2} \binom{n/2-1}{r/2-2} \psi^{ar/2}\, .
 \end{aligned}
 \end{equation}

\subsection{Computation of the number of components}\label{sec:number-components}

With the results obtained in the sections above, we can compute $c(\pi_1,\pi_2)=
|M_{n,m,r}^{\pi_1, \pi_2}|$, for $\gcd(n,m) = 1$.
% and $r\leq 4$. We
%start with the partitions $\pi = 
%\left\{1^r\right\}$, $\pi' = \left\{1^{r-2}, 2^1\right\}$. 

We start with the case $(\pi, \pi)$, where we have 
 \begin{align*}
  |M_{n,m,r}^{\pi, \pi}| &= \sum_{a=0}^{r-1} |N_{n,r}^\pi(\psi^a)|\, |N_{m,r}^\pi(\psi^a)| \\
  &= \sum_{e|n,e|r \atop f|m, f|r} 
  \frac{(-1)^{r/e+r/f}}{nm} \binom{n/e}{r/e}\binom{m/f}{r/f} 
  \sum_{k_1,k_2} \sum_{a=0}^{r-1} e^{2\pi i a \left(\frac{k_1}e+\frac{k_2}f\right)},
  \end{align*}
  where $k_1\in \ZZ_e^*$, $k_2\in \ZZ_f^*$. Now $e^{2\pi i \left(\frac{k_1}e+\frac{k_2}f\right)}=
   e^{2\pi i (k_1f+k_2e)/ef}$. 
  % is a non-trivial $r$-th root of unity unless $e=1$ and $f=1$,
  As $e|n, f|m$ and $\gcd(n,m)=1$, we have $\gcd(e,f)=1$. Hence $\ZZ_e^* \x \ZZ_f^* \cong \ZZ_{ef}^*$, via the map $(k_1,k_2)\mapsto
  k_1f+k_2e$. So the sum only contributes for $e=f=1$, and it gives
% \begin{align*}
%  |M_{n,m,r}^{\pi, \pi}| &= \sum_{a=0}^{r-1} |N_{n,r}^\pi(\psi^a)|\, |N_{m,r}^\pi(\psi^a)| \\
%  &= \sum_{e|n,e|r \atop f|m, f|r} 
%  \frac{(-1)^{r/e+r/f}}{nm} \binom{n/e}{r/e}\binom{m/f}{r/f} 
%  \sum_{k_1,k_2} \sum_{a=0}^{r-1} e^{2\pi i a \left(\frac{k_1}e+\frac{k_2}f\right)},
%  \end{align*}
%  
%  
%  By Remark \ref{rem:Mob1}, the sum  
%  $\sum_{k_1,k_2}e^{2\pi i (k_1f+k_2e)/ef}=\sum_{y|ef} \mu(y)\mathrm{SR}(ef/y,1)=\mu(ef)$.
%  So
% $$
%  |M_{n,m,r}^{\pi, \pi}|= \sum_{e|n,e|r \atop f|m, f|r}  \mu(ef) \frac{(-1)^{r/e+r/f}}{nm} \binom{n/e}{r/e}\binom{m/f}{r/f} .
%  $$
%This formula corrects the one given in \cite[Thm.~6.1]{MP}. %, which says that the character variety $\cR_r$ 
%When $\gcd(n,r)=1,\gcd(m,r)=1$, this becomes
\begin{equation}\label{eqn:MP-bien}
  |M_{n,m,r}^{\pi, \pi}|= \frac1{nm} \binom{n}{r}\binom{m}{r} r =\frac1r \binom{n-1}{r-1}\binom{m-1}{r-1}.
 \end{equation}
 
This formula agrees with \cite[Thm.~6.1]{MP}, which says that the character variety $ \cR_r$ 
has dimension $(r-1)^2$, and the number of 
irreducible components of this dimension is (\ref{eqn:MP-bien}). %  $\frac{1}{r} \binom{n-1}{r-1} \binom{m-1}{r-1}$.
  
 \begin{remark}
 We can get a formula when $\gcd(n,m)>1$ as well. We will not develop this.
 \end{remark}
 
For the case $(\pi,\pi')$, we use (\ref{eqn:LLL}) which says that $|N^{\pi'}_{m,r}|=-r|N^{\pi}_{m,r}|+\binom{m}{r-1}$, hence
 \begin{align*}
  |M_{n,m,r}^{\pi, \pi'}| &= \sum_{a=0}^{r-1} |N_{n,r}^{\pi}(\psi^a)|\, |N_{m,r}^{\pi'}(\psi^a)| 
   = -r |M_{n,m,r}^{\pi, \pi}| + \binom{m}{r-1} \sum_{a=0}^{r-1} |N_{n,r}^{\pi}(\psi^a)| .
   \end{align*}
Now use (\ref{eqn:Npi2}) and
 $\sum\limits_{a=0}^{r-1} \sum\limits_{\nu\in \mu_e^*} \nu^a=\sum\limits_{\nu\in \mu_e^*}  \sum\limits_{a=0}^{r-1} \nu^a =0$
 for $e>1$. Then
% 
% 
% r\,  \varphi(e),
% $$
% where $\varphi(e)$ is the Euler function. Then
\begin{equation}\label{eqn:MP-bien2}
  |M_{n,m,r}^{\pi, \pi'}| = -\binom{n-1}{r-1}\binom{m-1}{r-1} + \binom{m}{r-1} \frac{r}{n} \binom{n}{r}
   = \binom{n-1}{r-1} \binom{m-1}{r-2}.
\end{equation}

A similar argument shows the case $(\pi',\pi')$,
\begin{equation}\label{eqn:MP-bien3}
 \begin{aligned}
  |M_{n,m,r}^{\pi', \pi'}| =&\, r \binom{n-1}{r-1}\binom{m-1}{r-1} - r \binom{m}{r-1} \frac{r}{n} \binom{n}{r}
  - r \binom{n}{r-1} \frac{r}{m} \binom{m}{r}  \\ &\, + r \binom{n}{r-1}\binom{m}{r-1} 
   = r \binom{n-1}{r-2} \binom{m-1}{r-2}.
   \end{aligned}
   \end{equation}

Regarding the partition $\pi''$, we only write the formula for $(\pi,\pi'')$, as this is the only case that we need for $r=4$.
Using expression (\ref{eqn:pi''}) as before we get
\begin{equation}\label{eqn:MP-bien4}
  |M_{n,m,r}^{\pi, \pi''}| =\binom{n-1}{r-1} \binom{m-1}{r-3}.
  \end{equation}

For rank $r=4$, the only allowable partitions are 
$\pi=\{1^4\}$, $\pi'=\{1^2,2^1\}$, $\pi''=\{3^1,1^1\}$, $\pi'''=\{2^2\}$. All cases have been given
except $\pi'''=\{2^2\}$. For such case, $\mu_4=\{1,-1,i,-i\}$, and 
$$
 N^{\{2^2\}}_{n,4}(\varpi)=\{(\epsilon_1,\epsilon_2) | \epsilon_1\neq \epsilon_2, 
 \epsilon_1^n=\varpi, \epsilon_2^n=\varpi, \epsilon_1^2\epsilon_2^2=1\}.
$$
The map $(\epsilon_1,\epsilon_2) \mapsto (\epsilon_1^2,\epsilon_2^2)$ gives that 
 $$
  |N^{\{2^2\}}_{n,4}(\varpi)|= |N^{\{1^2\}}_{n,2}(\varpi^2)| +A_n(\varpi) , 
 $$
where 
$A_n(\varpi)$ counts the number of $\epsilon_1\neq \epsilon_2$ with $\epsilon_1^2=\epsilon^2_2$,   
 $\epsilon_1^n=\epsilon_2^n=\varpi$ and $\epsilon_1^2\epsilon_2^2=1$. This means that
 $(\epsilon_1,\epsilon_2)=(1,-1),(-1,1),(i,-i),(-i,i)$. Then
$$
A_n(\varpi)= \left\{ \begin{array}{ll}
 2, \qquad & n\equiv 2\pmod 4, \, \varpi=\pm 1 \\
 4, & n\equiv 0\pmod 4, \, \varpi=1  \\
 0, & \text{otherwise} \end{array}\right\} = \sum_{e|n, e|4} \sum_{s\in \ZZ_e^*} \psi^{as \, 4/e}\, .
 $$

We use this formula to compute $|M^{\pi_1,\pi'''}_{n,m,4}|$, for $\pi_1=\pi,\pi'$. Note that the case
$\pi_1=\pi'''$ is not needed due to Proposition \ref{prop:8.1}. After some manipulations,
\begin{equation}\label{eqn:MP-bien5}
\begin{aligned}
  |M_{n,m,4}^{\pi, \pi'''}| & =\frac12 \binom{n-1}{3} (m-1), \\
  |M_{n,m,4}^{\pi', \pi'''}| &= 2 \binom{n-1}{2} (m-1).
\end{aligned}
  \end{equation}

\subsection{Case of coprime indices}\label{subsec:coprime}

In this section we take a different approach to the computation of the number of components. This approach will be valid for any 
partitions $\pi_1, \pi_2$ of $r$ but, on the other hand, it will only work when we add the extra hypothesis 
$\gcd(n,r) = \gcd(m,r) = 1$.

Consider the group isomorphism $\ZZ_s \to \mu_s$, $k \mapsto e^{2\pi i k/s}$. Under this isomorphism, let $\varpi = e^{2\pi i k/r}$ 
for some $k \in \ZZ_r$. Introduce the sets
$$
	\tilde N_{n,r}(k, l) = \left\{(x_1, \ldots, x_r) \in \ZZ_{nr}^r \,|\, nx_i \equiv k \textrm{ (mod $r$)}, x_1 + \cdots + x_r \equiv l 
	\textrm{ (mod $n$)}\right\}.
$$
Then the set $\hat N_{n,r}(\varpi)$ in (\ref{eqn:Nn}) has the same cardinality as $\tilde N_{n,r}(k,0)$.

\begin{proposition}\label{prop:form-N}
If $\gcd(n,r)=1$ then
$$
	|\tilde N_{n,r}(k, l)| = |\left\{(\alpha_1, \ldots, \alpha_r) \in \ZZ_n^r\,|\, \alpha_1 + \cdots + \alpha_r = l\right\}|,
$$ for any $k \in \ZZ_r$ and $l \in \ZZ_n$.

\begin{proof}
Consider the homomorphism $\varphi: \ZZ_{nr}^r \to \ZZ_n \times \ZZ_r^r$,
given by $\varphi(x_1, \ldots, x_r) = ( x_1 + \cdots + x_r \textrm{ (mod $n$)}, nx_i \textrm{ (mod $r$)})$. We have that $\tilde N_{n,r}(k, 
l) = \varphi^{-1}(l, k, k, \ldots, k)$. In order to identify this map, observe that $\varphi$ factors as
\[
\begin{displaystyle}
   \xymatrix
   { \ZZ_{nr}^r \ar[r]^{\phi^r} \ar[rd]_{\varphi} & \ZZ_n^r \times \ZZ_r^r \ar[d]^{f \times g^r} \\
    & \ZZ_n \times \ZZ_r^r\, ,
   }
\end{displaystyle}   
\]
where the horizontal map is the isomorphism $\phi: \ZZ_{nr} \cong \ZZ_n \times \ZZ_r$ given by the map 
$x \mapsto (x \pmod n, x \pmod r)$. For the vertical arrow we have $f: \ZZ_n^r \to \ZZ_n$ given by $f(\alpha_1, \ldots, \alpha_r) = 
\alpha_1 + \ldots + \alpha_r$, and $g: \ZZ_r \to \ZZ_r$ is $g(\beta) = n\beta$. Therefore,
$$
	\phi^r(\tilde N_{n,r}(k,l)) = \left\{(\alpha_1, \ldots, \alpha_r) \in \ZZ_n^r\,|\, \alpha_1 + \cdots + \alpha_r = l\right\} \times
	\left\{(\beta_1, \ldots, \beta_r) \in \ZZ_r^r\,|\, n\beta_i = k\right\}.
$$
Moreover, since $n$ is invertible in $\ZZ_r$, the last factor is just a point. Hence
$$
	|\tilde N_{n,r}(k,l)| = |\phi^r(\tilde N_{n,r}(k,l))| = |\left\{(\alpha_1, \ldots, \alpha_r) \in \ZZ_n^r\,|\, \alpha_1 + \cdots + \alpha_r = 
	l\right\}|.
$$
\end{proof}

\end{proposition}

\begin{corollary}
If $\gcd(n,r)=1$, then $|\tilde N_{n,r}(k,l)|$ is independent of $k$ and $l$.

\begin{proof}
The independence of $k$ follows from Proposition \ref{prop:form-N}. For the independence of $l$, observe that the map
$(x_1, \ldots, x_r) \mapsto (x_1 + r^{-1}(l'-l), \ldots, x_r + r^{-1}(l'-l))$ is a bijection between $\tilde N_{n,r}(0, l)$ and
$\tilde N_{n,r}(0, l')$, for any $l,l' \in \ZZ_n$, where $r^{-1}$ is the inverse of $r \pmod n$, and lifted to $\ZZ_{nr}$.
\end{proof}
\end{corollary}

\begin{corollary}\label{cor:form-N-order}
If $\gcd(n,r)=1$, then
$$
	|\tilde N_{n,r}(k,l)| = n^{r-1}.
$$
\end{corollary}

\begin{proof}
Clearly, $\ZZ_n^r = \bigsqcup\limits_{l=0}^{n-1} \tilde N_{n,r}(0,l)$. Since all the elements in the decomposition 
have the same cardinality, we get that $n^r = |\ZZ_n^r| = \sum\limits_{l=0}^{n-1} \tilde N_{n,r}(0,l) = n |\tilde N_{n,r}(0,l)|$.
\end{proof}

\begin{corollary}\label{cor:form-N-final}
If $\gcd(n,r)=1$ then, for any partition $\pi = \left\{1^{e_1}, 2^{e_2}, \ldots, r^{e_r}\right\}$ we have that
$$
	|N_{n,r}^\pi(k,l)| = \frac{1}{n} \begin{pmatrix}
	n\\ e_1, e_2, \ldots, e_r \end{pmatrix} = \frac{(n-1)!}{e_1! e_2! \cdots e_r! (n - e_1 - \ldots - e_r)!}
.
$$
\begin{proof}
First of all, observe that the isomorphism of the proof of Proposition \ref{prop:form-N} preserves the strata $\tilde N_{n,r}^\pi(k,l)$, so 
$|\tilde N_{n,r}^\pi(k,l)|$ is independent of $k$ and $l$. It also commutes with the action of $S_r$ so
$|N_{n,r}^\pi(k,l)| = |\tilde N_{n,r}^\pi(k,l)/S_r|$ is also independent of $k$ and $l$.

At this point, the proof is analogous to the one of Corollary \ref{cor:form-N-order}, but now observe that $\bigsqcup\limits_{l=0}^{n-1} 
N_{n,r}^\pi(0,l)$ is the collection of tuples $(\alpha_1, \ldots, \alpha_r) \in \ZZ_n^r$, up to reordering, with $e_1$ collections of different elements, $e_2$ collections of $2$ equal elements and, in general, $e_i$ collections of $i$ equal elements, for
$1 \leq i \leq r$. The total count of these sets is the multinomial number
$\begin{pmatrix}
	n\\
	e_1, e_2, \ldots, e_r
\end{pmatrix}$,
and the result follows.
\end{proof}
\end{corollary}

With this result, we are ready to provide the final count.

\begin{theorem}\label{thm:coprime-components}
If $\gcd(n,r)=\gcd(m,r)=1$, then, for any $\pi= \big\{1^{e_1}, 2^{e_2}, \ldots, r^{e_r}\big\}$ and $\pi' = \big\{1^{e_1'},
 2^{e_2'}, \ldots, r^{e_r'}\big\}$, we have
$$
	|M_{n,m,r}^{\pi, \pi'}| = \frac{r}{nm} \begin{pmatrix}
	n\\
	e_1, e_2, \ldots, e_r \end{pmatrix}\begin{pmatrix}
	m\\
	e_1', e_2', \ldots, e_r'\end{pmatrix}
.
$$

\begin{proof}
Observe that $|M_{n,m,r}^{\pi, \pi'}| = \sum_{\varpi} |N_{n,r}^\pi(\varpi)|\, |N_{m,r}^{\pi'}(\varpi)|$. Hence, adding up and 
using that $|N_{n,r}^{\pi}(\varpi)|=|N_{n,r}^\pi(k,0)|$ is independent of $\varpi$, we get
$$
	|M_{n,m,r}^{\pi, \pi'}| = \sum_{\varpi \in \mu_r} |N_{n,r}^{\pi}(\varpi)| \cdot |N_{m,r}^{\pi'}(\varpi)| = r 
	|N_{n,r}^{\pi}(k, 0)| \cdot |N_{m,r}^{\pi'}(k, 0)|.
$$
Now the result follows from Corollary \ref{cor:form-N-final}.
\end{proof}

\end{theorem}

All formulas (\ref{eqn:MP-bien}), (\ref{eqn:MP-bien2}), (\ref{eqn:MP-bien3}), (\ref{eqn:MP-bien4}), (\ref{eqn:MP-bien5})
match exactly Theorem \ref{thm:coprime-components}. 
We conjecture that the formula of Theorem \ref{thm:coprime-components} also holds true for general $n,m$ with $\gcd(n,m)=1$. It is true for $r\leq 4$.

%%%%%%%%%%%%%%%%%%%%%%%%%%%%%%%%
\section{Irreducible character varieties for $\SL_2$ and $\SL_3$} \label{sec:rank-low}
%%%%%%%%%%%%%%%%%%%%%%%%%%%%%%%%

In this section, we will use the previous framework for computing the motive of 
the character variety of torus knots in the cases of rank $2$ and rank $3$. These results were first 
obtained in the papers \cite{Munoz} and \cite{MP}, respectively.

\subsection{Rank $2$}\label{sec:rank2}

By Remark \ref{rmk:reducible}, the only configuration of eigenvalues that admits irreducible representations is $\kappa = 
(\left\{\epsilon_1,\epsilon_2\right\}, \left\{\varepsilon_1,\varepsilon_2\right\})$ with $\epsilon_1 \neq \epsilon_2$ and $
\varepsilon_1 \neq \varepsilon_2$. To shorten the notation, for a type $\tau = (\xi, \sigma)$ with shape $\xi = \left(\left\{(d_{i,j}, m_{i,j})\right\}_{j=1}^{s_i}\right)_{i=1}^s$, we will denote by $\sigma_A$ and $\sigma_B$ the configuration of eigenvalues of each irreducible piece for $A$ and $B$ respectively.

In this case, $\cT_\kappa/S_\kappa$ contains the following two types of the form $\tau = (\xi, \sigma)$.

\begin{itemize}
	\item $\xi = (\left\{(1,1), (1,1)\right\})$ with eigenvalues per piece $\sigma_A = (\left\{\left\{\epsilon_1\right\}, \left\{\epsilon_2\right\}\right\})$ and $\sigma_B = (\left\{\left\{\varepsilon_1\right\}, \left\{\varepsilon_2\right\}\right\})$. For this type, we have
$$
	m_{\kappa}(\tau) = 2, \quad	[\frM_\tau^{\irr}] = 1, \quad [\cM_\tau] = 1, \quad [\cG_\tau] = (q-1)^2.
$$
Therefore $[R(\tau)] = 2 \frac{q(q+1)(q-1)^2}{(q-1)^2} = 2q^2+2q$.
	\item  $\xi = (\left\{(1,1)\right\}, \left\{(1,1)\right\})$ with eigenvalues per piece $\sigma_A = (\left\{\left\{\epsilon_1\right\}\right\}, \left\{\left\{\epsilon_2\right\}\right\})$ and $\sigma_B = (\left\{\left\{\varepsilon_1\right\}\right\}, \left\{\left\{\varepsilon_2\right\}\right\})$. In this case, we have
$$
	m_{\kappa}(\tau) = 4, \quad [\frM_\tau^{\irr}] = 1, \quad [\cM_\tau] = q-1, \quad [\cG_\tau] = (q-1)^2.
$$
Therefore $[R(\tau)] = 4 (q-1) \frac{q(q+1)(q-1)^2}{(q-1)^2} = 4q^3-4q$.
\end{itemize}

Putting this all together, we get that $[R_{\kappa}^{\red}] = 4q^3+2q^2-2q$. The total count of representations is $[R_{\kappa}] = 
\left( \frac{q(q-1)(q-1)^2}{(q-1)^2}\right)^2 = q^{4} + 2 q^{3} + q^{2}$. Therefore, the irreducible representations are 
$[R_{\kappa}^{\irr}] = [R_{\kappa}] - [R_{\kappa}^{\red}] = q^{4} - 2 q^{3} - q^{2} + 2 q$. Moreover, we
get that $[\frM_{\kappa}^{\irr}] = [R_{\kappa}^{\irr}]/[\PGL_2] = q-2$.

Finally, by (\ref{eqn:MP-bien})
we know how many configurations of eigenvalues $\kappa$ exist. This gives the number of components of $\frM^{\irr}_2$ of the form $\frM^{\irr}_\kappa$, so following (\ref{eq:m-kappa}) we get that
$$
	[\frM^{\irr}_{2}] = \frac{1}{2} \binom{n-1}{2-1} \binom{m-1}{2-1} (q-2) = \frac{1}{2}(n-1)(m-1) 
	(q-2).
$$

\subsection{Rank $3$} \label{sec:rank3}

In this case, there are three different types of possible configurations of eigenvalues, according to their multiplicity. Observe that there cannot be a double eigenvalue both in the matrices $A$ and $B$ since, in that case, the representation is automatically reducible by Remark \ref{rmk:reducible}. In the following, $\epsilon_1, \epsilon_2, \epsilon_3$ (resp.\ $\varepsilon_1, \varepsilon_2, \varepsilon_3$) will denote three different eigenvalues.

%With respect to $\cI_0(\tau)$ for a type $\tau$, observe that it is the fiber at $\Sigma_{\bm{\varepsilon}}$ of the map $\cI(\tau) \to [\Sigma_{\bm{\varepsilon}}]$ given by $\varrho=(A,B) \mapsto B$. Hence, $[\cI_0(\tau)]$ is a product in which each $1$-dimensional irreducible piece contributes with $1$ and each $2$-dimensional piece contibutes with $[R_\kappa^{\irr}]/[\Sigma_{\bm{\varepsilon}}] = (q-2)(q-1)$.

\subsubsection{$\kappa_1 = ((\epsilon_1, \epsilon_2, \epsilon_3), (\varepsilon_1, \varepsilon_2, \varepsilon_3))$} In this case, the three eigenvalues are different. This implies that $\cG_\tau$ is a torus for any type $\tau \in \cT_{\kappa_1}$ of rank equal to
the number of irreducible pieces. In such manner, the contributions of the possible types are shown in the following table.

\begin{center}
\small
\begin{tabular}{|c|c|c|c|c|c|c|}
\hline
	$\tau$ & $\cM_\tau$ & $\cG_\tau$ & $\frM_\tau^{\irr}$ & $m_\kappa(\tau)$\\
\hline
	\begin{tabular}{c}
$\xi = (\left\{(1,1),(2,1)\right\})$\\
$\sigma_A = (\left\{\left\{\epsilon_1\right\},\left\{\epsilon_2,\epsilon_3\right\}\right\})$ \\
$\sigma_B = (\left\{\left\{\varepsilon_1\right\},\left\{\varepsilon_2,\varepsilon_3\right\}\right\})$
\end{tabular} & 1 & $(q-1)^2$ & $q-2$ & $9$\\
\hline
	\begin{tabular}{c}
$\xi = (\left\{(1,1)\right\},\left\{(2,1)\right\})$\\
$\sigma_A = (\left\{\left\{\epsilon_1\right\}\right\},\left\{\left\{\epsilon_2,\epsilon_3\right\}\right\})$ \\
$\sigma_B = (\left\{\left\{\varepsilon_1\right\}\right\},\left\{\left\{\varepsilon_2,\varepsilon_3\right\}\right\})$
\end{tabular} & $q^2-1$ & $(q-1)^2$ & $q-2$ & $9$\\
\hline
	\begin{tabular}{c}
$\xi = (\left\{(2,1)\right\},\left\{(1,1)\right\})$\\
$\sigma_A = (\left\{\left\{\epsilon_1,\epsilon_2\right\}\right\}, \left\{\left\{\epsilon_3\right\}\right\})$ \\
$\sigma_B = (\left\{\left\{\varepsilon_1,\varepsilon_2\right\}\right\},\left\{\left\{\varepsilon_3\right\}\right\})$
\end{tabular} & $q^2-1$ & $(q-1)^2$ & $q-2$ & $9$\\
\hline
	\begin{tabular}{c}
$\xi = (\left\{(1,1),(1,1),(1,1)\right\})$\\
$\sigma_A = (\left\{\left\{\epsilon_1\right\},\left\{\epsilon_2\right\},\left\{\epsilon_3\right\}\right\})$ \\
$\sigma_B = (\left\{\left\{\varepsilon_1\right\},\left\{\varepsilon_2\right\},\left\{\varepsilon_3\right\}\right\})$
\end{tabular} & 1 & $(q-1)^3$ & $1$ & $6$\\
\hline
	\begin{tabular}{c}
$\xi = (\left\{(1,1),(1,1)\right\},\left\{(1,1)\right\})$\\
$\sigma_A = (\left\{\left\{\epsilon_1\right\},\left\{\epsilon_2\right\}\right\},\left\{\left\{\epsilon_3\right\}\right\})$ \\
$\sigma_B = (\left\{\left\{\varepsilon_1\right\},\left\{\varepsilon_2\right\}\right\},\left\{\left\{\varepsilon_3\right\}\right\})$
\end{tabular} & $q^2-1$ & $(q-1)^3$ & $1$ & $18$\\
\hline
	\begin{tabular}{c}
$\xi = (\left\{(1,1)\right\},\left\{(1,1),(1,1)\right\})$\\
$\sigma_A = (\left\{\left\{\epsilon_1\right\}\right\},\left\{\left\{\epsilon_2\right\},\left\{\epsilon_3\right\}\right\})$ \\
$\sigma_B = (\left\{\left\{\varepsilon_1\right\}\right\},\left\{\left\{\varepsilon_2\right\},\left\{\varepsilon_3\right\}\right\})$
\end{tabular} & $(q-1)^2$ & $(q-1)^3$ & $1$ & $18$\\
\hline
	\begin{tabular}{c}
$\xi = (\left\{(1,1)\right\},\left\{(1,1)\right\},\left\{(1,1)\right\})$\\
$\sigma_A = (\left\{\left\{\epsilon_1\right\}\right\},\left\{\left\{\epsilon_2\right\}\right\},\left\{\left\{\epsilon_3\right\}\right\})$ \\
$\sigma_B = (\left\{\left\{\varepsilon_1\right\}\right\},\left\{\left\{\varepsilon_2\right\}\right\},\left\{\left\{\varepsilon_3\right\}\right\})$
\end{tabular} & $q(q-1)^2$ & $(q-1)^3$ & $1$ & $36$\\
\hline
\end{tabular}
\end{center}

Adding up all the contributions we get
$$
	[R_{\kappa_1}^{\red}] = 18  q^{10} + 18  q^{9} - 9  q^{8} - 9  q^{7} + 6 q^{6} + 21  q^{5} + 3  q^{4} - 12 q^{3}.
$$
Moreover, the total set of representations has motive
$$
	[R_{\kappa_1}] =q^{12} + 4 q^{11} + 8 q^{10} + 10 q^{9} + 8 q^{8} + 4  q^{7} + q^{6} ,
$$
and therefore, we get
 \begin{align*}
   [R_{\kappa_1}^{\irr}] &= q^{12} + 4 q^{11} -10 q^{10} -8 q^{9} + 17 q^{8} + 13  q^{7} -5 q^{6} 
  -21  q^{5} - 3  q^{4} + 12 q^{3}, \\
 [\frM_{\kappa_1}^{\irr}] &= [R_{\kappa_1}^{\irr}]/[\PGL_3] = q^{4} + 4 q^{3} - 9  q^{2} - 3 q + 12.
 \end{align*}
 
 \begin{remark}
 This corrects the formula in the published version of \cite[Thm.\ 8.3]{MP}. The current arXiv version contains the corrected formula,
 which matches this one.
 \end{remark}

\subsubsection{${\kappa_2} = ((\epsilon_1, \epsilon_2, \epsilon_3), (\varepsilon_1, \varepsilon_1, \varepsilon_2))$} \label{sec:rank3-repeatedB}
Again, in this case the matrix $A$ has no repeated eigenvalues, so the gauge group $\cG_\tau$ is again a torus for all $\tau \in \cT_{\kappa_2}$. Nevertheless, now $\cM_\tau$ may vary since we need to take into account the repeated eigenvalues of the matrix $B$.

\begin{center}
\small
\begin{tabular}{|c|c|c|c|c|c|c|}
\hline
	$\tau$ & $\cM_\tau$ & $\cG_\tau$ & $\frM_\tau^{\irr}$ & $m_\kappa(\tau)$\\
\hline
	\begin{tabular}{c}
$\xi = (\left\{(1,1),(2,1)\right\})$\\
$\sigma_A = (\left\{\left\{\epsilon_1\right\},\left\{\epsilon_2,\epsilon_3\right\}\right\})$ \\
$\sigma_B = (\left\{\left\{\varepsilon_1\right\},\left\{\varepsilon_1,\varepsilon_2\right\}\right\})$
\end{tabular} & 1 & $(q-1)^2$ & $q-2$ & $3$\\
\hline
	\begin{tabular}{c}
$\xi = (\left\{(1,1)\right\},\left\{(2,1)\right\})$\\
$\sigma_A = (\left\{\left\{\epsilon_1\right\}\right\},\left\{\left\{\epsilon_2,\epsilon_3\right\}\right\})$ \\
$\sigma_B = (\left\{\left\{\varepsilon_1\right\}\right\},\left\{\left\{\varepsilon_1,\varepsilon_2\right\}\right\})$
\end{tabular} & $q-1$ & $(q-1)^2$ & $q-2$ & $3$\\
\hline
	\begin{tabular}{c}
$\xi = (\left\{(2,1)\right\},\left\{(1,1)\right\})$\\
$\sigma_A = (\left\{\left\{\epsilon_1,\epsilon_2\right\}\right\}, \left\{\left\{\epsilon_3\right\}\right\})$ \\
$\sigma_B = (\left\{\left\{\varepsilon_1,\varepsilon_2\right\}\right\},\left\{\left\{\varepsilon_1\right\}\right\})$
\end{tabular} & $q-1$ & $(q-1)^2$ & $q-2$ & $3$\\
\hline
	\begin{tabular}{c}
$\xi = (\left\{(1,1),(1,1),(1,1)\right\})$\\
$\sigma_A = (\left\{\left\{\epsilon_1\right\},\left\{\epsilon_2\right\},\left\{\epsilon_3\right\}\right\})$ \\
$\sigma_B = (\left\{\left\{\varepsilon_1\right\},\left\{\varepsilon_1\right\},\left\{\varepsilon_2\right\}\right\})$
\end{tabular} & 1 & $(q-1)^3$ & $1$ & $3$\\
\hline
	\begin{tabular}{c}
$\xi = (\left\{(1,1),(1,1)\right\},\left\{(1,1)\right\})$\\
$\sigma_A = (\left\{\left\{\epsilon_1\right\},\left\{\epsilon_2\right\}\right\},\left\{\left\{\epsilon_3\right\}\right\})$ \\
$\sigma_B = (\left\{\left\{\varepsilon_1\right\},\left\{\varepsilon_1\right\}\right\},\left\{\left\{\varepsilon_2\right\}\right\})$
\end{tabular} & $q^2-1$ & $(q-1)^3$ & $1$ & $3$\\
\hline
	\begin{tabular}{c}
$\xi = (\left\{(1,1),(1,1)\right\},\left\{(1,1)\right\})$\\
$\sigma_A = (\left\{\left\{\epsilon_1\right\},\left\{\epsilon_2\right\}\right\},\left\{\left\{\epsilon_3\right\}\right\})$ \\
$\sigma_B = (\left\{\left\{\varepsilon_1\right\},\left\{\varepsilon_2\right\}\right\},\left\{\left\{\varepsilon_1\right\}\right\})$
\end{tabular} & $q-1$ & $(q-1)^3$ & $1$ & $6$\\
\hline
	\begin{tabular}{c}
$\xi = (\left\{(1,1)\right\},\left\{(1,1),(1,1)\right\})$\\
$\sigma_A = (\left\{\left\{\epsilon_1\right\}\right\},\left\{\left\{\epsilon_2\right\},\left\{\epsilon_3\right\}\right\})$ \\
$\sigma_B = (\left\{\left\{\varepsilon_2\right\}\right\},\left\{\left\{\varepsilon_1\right\},\left\{\varepsilon_1\right\}\right\})$
\end{tabular} & $(q-1)^2$ & $(q-1)^3$ & $1$ & $3$\\
\hline
	\begin{tabular}{c}
$\xi = (\left\{(1,1)\right\},\left\{(1,1)\right\},\left\{(1,1)\right\})$\\
$\sigma_A = (\left\{\left\{\epsilon_1\right\}\right\},\left\{\left\{\epsilon_2\right\}\right\},\left\{\left\{\epsilon_3\right\}\right\})$ \\
$\sigma_B = (\left\{\left\{\varepsilon_1\right\}\right\},\left\{\left\{\varepsilon_2\right\}\right\},\left\{\left\{\varepsilon_1\right\}\right\})$
\end{tabular} & $(q-1)^2$ & $(q-1)^3$ & $1$ & $6$\\
\hline
\end{tabular}
\end{center}

Adding up all the contributions we get
$$
	[R_{\kappa_2}^{\red}] = 6 q^{9} + 3  q^{8} + 3  q^{7} + 3  q^{6} + 3 q^{5} + 3 q^{4} - 3 q^{3}.
$$
Moreover, the total set of representations has motive
$$
	[R_{\kappa_2}] =q^{10} + 3 q^{9} + 5 q^{8} + 5 q^{7} + 3 q^{6} + q^{5},
$$
and, therefore, we get
 \begin{align*}
   [R_{\kappa_2}^{\irr}] &= q^{10} - 3 q^{9} + 2 q^{8} + 2 q^{7} -2 q^{5}-3q^4+3q^3, \\
 [\frM_{\kappa_2}^{\irr}] &= [R_{\kappa_2}^{\irr}]/[\PGL_3] = q^{2} - 3 q + 3.
 \end{align*}

\subsubsection{${\kappa_3} = ((\epsilon_1, \epsilon_1, \epsilon_2), (\varepsilon_1, \varepsilon_2, \varepsilon_3))$} In this case, there are repeated eigenvalues in $A$. This implies that now the gauge group will contain lines accounting for the repeated eigenvalues in $A$.

\begin{center}
\small
\begin{tabular}{|c|c|c|c|c|c|c|}
\hline
	$\tau$ & $\cM_\tau$ & $\cG_\tau$ & $\frM_\tau^{\irr}$ & $m_\kappa(\tau)$\\
\hline
	\begin{tabular}{c}
$\xi = (\left\{(1,1),(2,1)\right\})$\\
$\sigma_A = (\left\{\left\{\epsilon_1\right\},\left\{\epsilon_1,\epsilon_2\right\}\right\})$ \\
$\sigma_B = (\left\{\left\{\varepsilon_1\right\},\left\{\varepsilon_2,\varepsilon_3\right\}\right\})$
\end{tabular} & 1 & $(q-1)^2$ & $q-2$ & $3$\\
\hline
	\begin{tabular}{c}
$\xi = (\left\{(1,1)\right\},\left\{(2,1)\right\})$\\
$\sigma_A = (\left\{\left\{\epsilon_1\right\}\right\},\left\{\left\{\epsilon_1,\epsilon_2\right\}\right\})$ \\
$\sigma_B = (\left\{\left\{\varepsilon_1\right\}\right\},\left\{\left\{\varepsilon_2,\varepsilon_3\right\}\right\})$
\end{tabular} & $q^2-q$ & $q(q-1)^2$ & $q-2$ & $3$\\
\hline
	\begin{tabular}{c}
$\xi = (\left\{(2,1)\right\},\left\{(1,1)\right\})$\\
$\sigma_A = (\left\{\left\{\epsilon_1,\epsilon_2\right\}\right\}, \left\{\left\{\epsilon_1\right\}\right\})$ \\
$\sigma_B = (\left\{\left\{\varepsilon_1,\varepsilon_2\right\}\right\},\left\{\left\{\varepsilon_3\right\}\right\})$
\end{tabular} & $q^2-q$ & $q(q-1)^2$ & $q-2$ & $3$\\
\hline
	\begin{tabular}{c}
$\xi = (\left\{(1,1),(1,1),(1,1)\right\})$\\
$\sigma_A = (\left\{\left\{\epsilon_1\right\},\left\{\epsilon_1\right\},\left\{\epsilon_2\right\}\right\})$ \\
$\sigma_B = (\left\{\left\{\varepsilon_1\right\},\left\{\varepsilon_2\right\},\left\{\varepsilon_3\right\}\right\})$
\end{tabular} & 1 & $(q-1)^3$ & $1$ & $3$\\
\hline
	\begin{tabular}{c}
$\xi = (\left\{(1,1),(1,1)\right\},\left\{(1,1)\right\})$\\
$\sigma_A = (\left\{\left\{\epsilon_1\right\},\left\{\epsilon_1\right\}\right\},\left\{\left\{\epsilon_2\right\}\right\})$ \\
$\sigma_B = (\left\{\left\{\varepsilon_1\right\},\left\{\varepsilon_2\right\}\right\},\left\{\left\{\varepsilon_3\right\}\right\})$
\end{tabular} & $q^2-1$ & $(q-1)^3$ & $1$ & $3$\\
\hline
	\begin{tabular}{c}
$\xi = (\left\{(1,1),(1,1)\right\},\left\{(1,1)\right\})$\\
$\sigma_A = (\left\{\left\{\epsilon_1\right\},\left\{\epsilon_2\right\}\right\},\left\{\left\{\epsilon_1\right\}\right\})$ \\
$\sigma_B = (\left\{\left\{\varepsilon_1\right\},\left\{\varepsilon_2\right\}\right\},\left\{\left\{\varepsilon_3\right\}\right\})$
\end{tabular} & $q^2-q$ & $q(q-1)^3$ & $1$ & $6$\\
\hline
	\begin{tabular}{c}
$\xi = (\left\{(1,1)\right\},\left\{(1,1),(1,1)\right\})$\\
$\sigma_A = (\left\{\left\{\epsilon_2\right\}\right\},\left\{\left\{\epsilon_1\right\},\left\{\epsilon_1\right\}\right\})$ \\
$\sigma_B = (\left\{\left\{\varepsilon_1\right\}\right\},\left\{\left\{\varepsilon_2\right\},\left\{\varepsilon_3\right\}\right\})$
\end{tabular} & $(q-1)^2$ & $(q-1)^3$ & $1$ & $3$\\
\hline
	\begin{tabular}{c}
$\xi = (\left\{(1,1)\right\},\left\{(1,1)\right\},\left\{(1,1)\right\})$\\
$\sigma_A = (\left\{\left\{\epsilon_1\right\}\right\},\left\{\left\{\epsilon_2\right\}\right\},\left\{\left\{\epsilon_1\right\}\right\})$ \\
$\sigma_B = (\left\{\left\{\varepsilon_1\right\}\right\},\left\{\left\{\varepsilon_2\right\}\right\},\left\{\left\{\varepsilon_3\right\}\right\})$
\end{tabular} & $q(q-1)^2$ & $q(q-1)^3$ & $1$ & $6$\\
\hline
\end{tabular}
\end{center}

Adding up all the contributions we get
$$
	[R_{\kappa_3}^{\red}] = 6 q^{9} + 3  q^{8} + 3  q^{7} + 3  q^{6} + 3 q^{5} + 3 q^{4} - 3 q^{3}.
$$
Moreover, the total set of representations has motive
$$
	[R_{\kappa_3}] = q^{10} + 3 q^{9} + 5 q^{8} + 5 q^{7} + 3 q^{6} + q^{5},
$$
and, therefore, we get
 \begin{align*}
   [R_{\kappa_3}^{\irr}] &= q^{10} - 3 q^{9} + 2 q^{8} + 2 q^{7} -2 q^{5}-3q^4+3q^3, \\
 [\frM_{\kappa_3}^{\irr}] &= [R_{\kappa_3}^{\irr}]/[\PGL_3] = q^{2} - 3 q + 3.
 \end{align*}
Observe that this result agrees with the one of Section \ref{sec:rank3-repeatedB}. This is not a coincidence, since the role of $A$ and $B$ are interchangeable ($n$ and $m$ play no role in the computation) so both cases are computing the same component.

In order to put all the computations together, let $\pi_0 = \left\{1^3\right\}$ be the partition into three different eigenvalues and $\pi_1 = \left\{1^1, 2^1\right\}$ be the partition with two equal eigenvalues. Then, we have that
\begin{align*}
	[\frM^{\irr}_3] &= |M_{n,m,3}^{\pi_0, \pi_0}| [\frM_{\kappa_1}^{\irr}] + |M_{n,m,3}^{\pi_0, \pi_1}| [\frM_{\kappa_2}^{\irr}] + |M_{n,m,3}^{\pi_1, \pi_0}| [\frM_{\kappa_3}^{\irr}] \\
	 &= \frac{3}{nm} \begin{pmatrix}
	n\\
	3 \end{pmatrix} \begin{pmatrix}
	m\\
	3\end{pmatrix} (q^{4} + 4 q^{3} - 9  q^{2} - 3 q + 12) \\
	&\qquad + \frac{3}{nm}\left( \begin{pmatrix}
	n\\
	3 \end{pmatrix}\begin{pmatrix}
	m\\
	1,1\end{pmatrix} + \begin{pmatrix}
	n\\
	1,1 \end{pmatrix}\begin{pmatrix}
	m\\
	3\end{pmatrix}\right) (q^{2} - 3 q + 3) \\
	&=  \frac{(n-1)(n-2)(m-1)(m-2)}{12} (q^{4} + 4 q^{3} - 9  q^{2} - 3 q + 12) \\
	& \qquad + \frac{(n-1)(m-1)(n+m-4)}{2} (q^{2} - 3 q + 3),
\end{align*}
using the formula for the coefficients $|M_{n,m,3}^{\pi_1, \pi_0}|,|M_{n,m,3}^{\pi_1, \pi_0}|$ and $|M_{n,m,3}^{\pi_0, \pi_1}|$ 
in Section \ref{sec:number-components}. 

%%%%%%%%%%%%%%%%%%%%%%%%%%%%%%%%%%%%%%%%%
\section{Irreducible character variety for $\SL_4$} \label{sec:rank4}
%%%%%%%%%%%%%%%%%%%%%%%%%%%%%%%%%%%%%%%%%

In this section, we compute the motive of the irreducible $\SL_4$-character variety. As in Section \ref{sec:rank-low}, with the methods 
developed in this paper 
the computation reduces to an enumerative problem. First, we need to consider the possible eigenvalues configurations that may contain irreducible representations, in the sense that they do not fulfill the conditions of Remark \ref{rmk:reducible}. As we will see, there are $11$ possible configurations.

Among these, there is a special configuration, namely $\kappa = ((\epsilon_1,\epsilon_1, \epsilon_2, \epsilon_2), (\varepsilon_1,\varepsilon_1, \varepsilon_2, \varepsilon_2))$. This configuration is the only one that contains types $\tau \in \cT_\kappa$ with isotypic components of dimension greater than $1$ and with multiplicity greater than $1$, namely, the type $\tau=(\xi,\sigma)$ 
with shape $\xi = (\left\{(2, 2)\right\})$ and eigenvalues $\sigma_A = (\left\{\left\{\epsilon_1, \epsilon_2\right\}\right\})$ and $\sigma_B = (\left\{\left\{\varepsilon_1, \varepsilon_2\right\}\right\})$. Hence, this is the only configuration for which formula (\ref{eqn:working-eqn3}) is not valid because there is an action of $\ZZ_2$ on $\frM_\tau^{\irr}$. Fortunately, we do not need to count this stratum thanks to the following result.

\begin{proposition}\label{prop:8.1}
Let $\kappa = ((\epsilon_1,\epsilon_1, \epsilon_2, \epsilon_2), (\varepsilon_1,\varepsilon_1, \varepsilon_2, \varepsilon_2))$. Then $R_\kappa^{\irr} = \emptyset$, that is, all the representations of $R_\kappa$ are reducible.
\end{proposition}

\begin{proof}
Let $\rho = (A, B) \in R_\kappa$ and let $V_i, W_i \subseteq k^4$ be the eigenspaces of $A$ and $B$ of eigenvalues $\epsilon_i$ and $\varepsilon_i$ respectively, for $i = 1,2$. We claim that there exists $v_1 \in V_1$, $v_2 \in V_2$ and $\mu \in k$ such that $v_1 + v_2 \in W_1$ and $v_1 + \mu v_2 \in W_2$. In that case, the subspace $H = \langle v_1, v_2\rangle = \langle v_1 + v_2, v_1 + \mu v_2\rangle$ is an invariant subrepresentation, proving that $\rho$ is reducible. Observe that we must have $\mu \neq 1$ since otherwise $v_1 + v_2 \in W_1 \cap W_2$.

In order to obtain these vectors, let $L_i = L_{i,1} + L_{i,2}: V_1 \oplus V_2 \to k^2$ be a linear map with $W_i = \ker L_i$ for $i = 1,2$. Observe that we can suppose that the maps $L_{i,j}$ are invertible since otherwise $W_i \cap V_j \neq 0$. The vectors we are looking for must satisfy 
$L_{1,1} (v_1) + L_{1,2} (v_2) = 0$ and $L_{2,1}( v_1) + \mu L_{2,2} (v_2) = 0$. Form the matrix
$$
	L = \begin{pmatrix}
		L_{1,1} & L_{1,2} \\
		L_{2,1} & \mu L_{2,2}
	\end{pmatrix}.
$$
The equation $\det L = 0$ gives a quadratic equation for $\mu$. Choosing one of its solutions, 
any $v_1 \oplus v_2 \in \ker L$ provides the desired vectors.
\end{proof}

With this result at hand, it is enough to consider the remaining $10$ possible configurations of eigenvalues. As we will see, depending on whether they admit types with isotypic components of higher multiplicity or not, the counting method is slightly different, so we will analyze them in separate sections.

Observe that, as suggested for the computations of the cases of rank $2$ and rank $3$, the number of types to be analyzed grows exponentially with the rank. In the rank $2$ case, we needed to consider $2$ classes of types, whereas in the rank $3$ case we had to deal with $23$ different classes of types. In the rank $4$ case, there are more than $350$ types to be analyzed, so the computation must be performed with the aid of a computer algebra system. In our case, we use SageMath and the complete script performing the calculating can be found in \cite{GPMScript}. For this reason, in order to keep this paper at a reasonable size, we will only report here some particular cases in order to illustrate the method. For a complete description of the possible types and their count, please refer to the script \cite{GPMScript} and the output therein.

\subsection{Configurations with all the isotypic components of multiplicity $1$} \label{subsec:8.1}

This case is completely analogous to the rank $3$ case, as described in Section \ref{sec:rank3}. Since there are no isotypic components of higher multiplicity, the motive of the variety $\cM_\tau$ and of the gauge group $\cG_\tau$ can be directly computed using (\ref{eq:mult-1}) and (\ref{eq:gauge-group}), which are purely combinatorial formulas depending on the coincidence of eigenvalues on the different blocks of $A$ and $B$.

There are seven configurations of eigenvalues with this property, whose count is as follows. Here eigenvalues with different subindices are distinct numbers.

\begin{itemize}
	\item $\kappa = ((\epsilon_1,\epsilon_1,\epsilon_1,\epsilon_2), (\varepsilon_1,\varepsilon_2,\varepsilon_3,\varepsilon_4))$,
\begin{align*}
	[R_\kappa^{\irr}] = &\, q^{18} - 4 \, q^{17} + 5 \, q^{16} - q^{15} - 3 \, q^{14} + 3 \, q^{13} \\
	&- 5 \, q^{12} + 7 \, q^{11} - 2 \, q^{10} + q^{9} - 6 \, q^{7} + 4 \, q^{6}.
\end{align*}

	\item $\kappa = ((\epsilon_1,\epsilon_1,\epsilon_2,\epsilon_2),(\varepsilon_1,\varepsilon_2,\varepsilon_3,\varepsilon_4))$,
\begin{align*}
	 [R_\kappa^{\irr}] = &\, q^{20} + 4 \, q^{19} - 12 \, q^{18} - 4 \, q^{17} + 24 \, q^{16} - 11 \, q^{15} - 3 \, q^{14} - 7 \, q^{13} \\
	 &- 6 \, q^{12} + 25 \, q^{11} - 3 \, q^{10} + 11 \, q^{9} - 19 \, q^{8} - 18 \, q^{7} + 18 \, q^{6}. 
\end{align*}

	\item $\kappa = ((\epsilon_1,\epsilon_1,\epsilon_2,\epsilon_3),(\varepsilon_1,\varepsilon_2,\varepsilon_3,\varepsilon_4))$,
\begin{align*}
	 [R_\kappa^{\irr}] = &\, q^{22} + 5 \, q^{21} + 6 \, q^{20} - 40 \, q^{19} - 13 \, q^{18} + 57 \, q^{17} + 51 \, q^{16} - 35 \, q^{15} - 74 \, q^{14} \\
	 &- 32 \, q^{13} + 25 \, q^{12} + 93 \, q^{11} + 38 \, q^{10} - 30 \, q^{9} - 82 \, q^{8} - 18 \, q^{7} + 48 \, q^{6}. 
\end{align*}

	\item $\kappa = ((\epsilon_1,\epsilon_2,\epsilon_3,\epsilon_4),(\varepsilon_1,\varepsilon_1,\varepsilon_1,\varepsilon_2))$,
\begin{align*}
	 [R_\kappa^{\irr}] = &\, q^{18} - 4 \, q^{17} + 5 \, q^{16} - q^{15} - 3 \, q^{14}\\
	 &+ 3 \, q^{13} - 5 \, q^{12} + 7 \, q^{11} - 2 \, q^{10} + q^{9} - 6 \, q^{7} + 4 \, q^{6}. 
\end{align*}

	\item $\kappa = ((\epsilon_1,\epsilon_2,\epsilon_3,\epsilon_4),(\varepsilon_1,\varepsilon_1,\varepsilon_2,\varepsilon_2))$,
\begin{align*}
	 [R_\kappa^{\irr}] = &\, q^{20} + 4 \, q^{19} - 12 \, q^{18} - 4 \, q^{17} + 24 \, q^{16} - 11 \, q^{15} - 3 \, q^{14} - 7 \, q^{13} \\
	 &- 6 \, q^{12} + 25 \, q^{11} - 3 \, q^{10} + 11 \, q^{9} - 19 \, q^{8} - 18 \, q^{7} + 18 \, q^{6}. 
\end{align*}

	\item $\kappa = ((\epsilon_1,\epsilon_2,\epsilon_3,\epsilon_4),(\varepsilon_1,\varepsilon_1,\varepsilon_2,\varepsilon_3))$,
\begin{align*}
	 [R_\kappa^{\irr}] = &\, q^{22} + 5 \, q^{21} + 6 \, q^{20} - 40 \, q^{19} - 13 \, q^{18} + 57 \, q^{17} + 51 \, q^{16} - 35 \, q^{15} - 74 \, q^{14}\\
	 & - 32 \, q^{13} + 25 \, q^{12} + 93 \, q^{11} + 38 \, q^{10} - 30 \, q^{9} - 82 \, q^{8} - 18 \, q^{7} + 48 \, q^{6}. 
\end{align*}

	\item $\kappa = ((\epsilon_1,\epsilon_2,\epsilon_3,\epsilon_4),(\varepsilon_1,\varepsilon_2,\varepsilon_3,\varepsilon_4))$,
\begin{align*}
	 [R_\kappa^{\irr}] = &\, q^{24} + 6 \, q^{23} + 19 \, q^{22} + 10 \, q^{21} - 125 \, q^{20} - 68 \, q^{19} + 106 \, q^{18} \\& + 260 \, q^{17} + 129 \, q^{16} - 344 \, q^{15} - 277 \, q^{14} - 88 \, q^{13} + 265 \, q^{12} \\
	 &+ 406 \, q^{11} + 8 \, q^{10} - 182 \, q^{9} - 270 \, q^{8} + 144 \, q^{6}. 
\end{align*}
\end{itemize}

\subsection{Configurations with isotypic components of higher multiplicity} \label{subsec:8.2}

In this case, the configurations of eigenvalues admit types with isotypic components $W_{i,j}$ of multiplicity $m_{i,j} > 1$. Observe that, since the configuration $\kappa = ((\epsilon_1,\epsilon_1, \epsilon_2, \epsilon_2), (\varepsilon_1,\varepsilon_1, \varepsilon_2, \varepsilon_2))$ has been excluded, the only possibility that may occur is $\dim W_{i,j} = 1$. Moreover, if $m_{i,j} > 2$ then there would exist a common eigenvector to $A$ and $B$ (c.f.\ Remark \ref{rmk:reducible}).

Therefore, these higher multiplicity blocks must have $\dim W_{i,j} = 1$ and $m_{i,j} = 2$. In this way, the motive of the associated variety $[\cM_\tau]$ can be computed using items (\ref{item:mult-1}) and (\ref{item:mult-2}) of Section \ref{sec:explicit-formulas}. As always, the calculation of the motive of the gauge group, $[\cG_\tau]$, as well as the multiplicities $m_\kappa(\tau)$ are a purely combinatorial matter. Finally, the irreducible part $[\frM_\tau^{\irr}]$ can be obtained from the results of Section \ref{sec:rank-low}.

Therefore, we get three configuration of eigenvalues, whose contributions are as follows:

\begin{itemize}
	\item $\kappa = ((\epsilon_1,\epsilon_1,\epsilon_2,\epsilon_2),(\varepsilon_1,\varepsilon_1,\varepsilon_2,\varepsilon_3))$,
\begin{align*}
	 [R_\kappa^{\irr}] = &\, q^{18} - 3 \, q^{17} + 4 \, q^{16} - 2 \, q^{15} - 3 \, q^{14} + 3 \, q^{13} \\
	 &- 3 \, q^{12} + 7 \, q^{11} - 2 \, q^{10} - q^{8} - 5 \, q^{7} + 4 \, q^{6}. 
\end{align*}
	\item $\kappa = ((\epsilon_1,\epsilon_1,\epsilon_2,\epsilon_3)$ $(\varepsilon_1,\varepsilon_1,\varepsilon_2,\varepsilon_2))$,
\begin{align*}
	 [R_\kappa^{\irr}] = &\,q^{18} - 3 \, q^{17} + 4 \, q^{16} - 2 \, q^{15} - 3 \, q^{14} + 3 \, q^{13} \\
	 &- 3 \, q^{12} + 7 \, q^{11} - 2 \, q^{10} - q^{8} - 5 \, q^{7} + 4 \, q^{6}. 
\end{align*}
	\item $\kappa = ((\epsilon_1,\epsilon_1,\epsilon_2,\epsilon_3),(\varepsilon_1,\varepsilon_1,\varepsilon_2,\varepsilon_3))$,
	\begin{align*}
	 [R_\kappa^{\irr}] = &\,q^{20} + 2 \, q^{19} - 11 \, q^{18} + 4 \, q^{17} + 18 \, q^{16} - 15 \, q^{15} - 5 \, q^{14} - 8 \, q^{13} \\
	 &+ 5 \, q^{12} + 24 \, q^{11} - q^{10} + 4 \, q^{9} - 24 \, q^{8} - 11 \, q^{7} + 17 \, q^{6}.
\end{align*}
\end{itemize}

With these results, we can finally compute the character variety of irreducible components of torus knots in $\SL_4$. Consider the partitions $\pi_0 = \left\{1^4\right\}$, $\pi_1 = \left\{1^2, 2^1\right\}$, $\pi_2 = \left\{2^2\right\}$ and $\pi_3 = \left\{1^1, 3^1\right\}$ and, given two of such partitions, shorten $C_{\pi, \pi'} = |M_{n,m,4}^{\pi, \pi'}| + |M_{n,m,4}^{\pi', \pi}|$ if $\pi \neq \pi'$ and $C_{\pi, \pi} = |M_{n,m,4}^{\pi, \pi}|$. With this notation, the final result is
\begin{equation}\label{eqn:frM4irr}
\begin{aligned}
[\frM_4^{\irr}] =&\, C_{\pi_0, \pi_0} {\left(q^{9} + 6  q^{8} + 20  q^{7} + 17  q^{6} - 98  q^{5} - 26  q^{4} + 38  q^{3} + 126  q^{2} - 144\right)}  \\
&+ C_{\pi_0, \pi_1}{\left(q^{7} + 5  q^{6} + 7  q^{5} - 34  q^{4} + 34  q^{2} + 18  q - 48\right)}  \\
& +C_{\pi_0, \pi_2}{\left(q^{5} + 4  q^{4} - 11  q^{3} + q^{2} + 18  q - 18\right)}  + C_{\pi_0, \pi_3}{\left(q^{3} - 4  q^{2} + 6  q - 4\right)} \\
& + C_{\pi_1, \pi_2}{\left(q^{3} - 3  q^{2} + 5  q - 4\right)}  + C_{\pi_1, \pi_1}{\left(q^{5} + 2  q^{4} - 10  q^{3} + 7  q^{2} + 11  q - 17\right)} .
\end{aligned}
\end{equation}

The formulas for $C_{\pi_1,\pi_2}$ appear in Theorem \ref{thm:coprime-components} (and the comment following it for $r=4$).
This completes the proof of Theorem \ref{thm:main}.

%%%%%%%%%%%%%%%%%%%%%%%%%%%%%%%%%%%%%%%%
\section{Motive of the character varieties}\label{sec:motive-total}
%%%%%%%%%%%%%%%%%%%%%%%%%%%%%%%%%%%%%%%%

To finish this work, let us compute the motive of the character varieties 
 $$
 \frX_r=\frX(\G_{n,m},\SL_r),
 $$
or equivalently, of the moduli spaces of representations $\frM_r = \frM(\G_{n,m},\SL_r)$, for the torus knots $\G_{n,,m}$ and $r=4$.
We recall that, for lower rank, the motive of $[\frX_2]$ appears in 
\cite[Prop.\ 7.1]{MP} and the motive of $[\frX_3]$ is computed in \cite[Thm.\ 8.3]{MP}.

Recall that the space $\frX_r$ parametrizes (isomorphism classes of) semi-simple representations. 
Let $\Pi_r$ be the set of all partitions of $r$ and let us denote by 
$\pi=\{r_1^{a_1},\ldots, r_s^{a_s}\}$ a partition of $r$. Then there is a decomposition of the character variety 
 $$
 \frX_r=\bigsqcup_{\pi\in \Pi_r} \frX_r^{\pi}\, , 
 $$
where 
  \begin{align*} \frX_r^{\pi} &=\Big\{ \rho=\mathop{\oplus}\limits_{t=1}^s \left[\mathop{\oplus}\limits_{l=1}^{a_t}  \rho_{t,l}\right]\,\left|\,  
 \rho_{t,l} \in \tilde\frX^{\irr}_{r_t}, \det \rho=1\right.\Big\}  
 \subset \prod_{t=1}^s \Sym^{a_t} 
 \big(\tilde\frX^{\irr}_{r_t}\big).
  \end{align*}
Here, we set $\tilde\frX_r^{\irr}=\frX^{\irr}(\G_{n,m},\GL_r)$ and $\bar\frX_r^{\irr}=\frX^{\irr}(\G_{n,m},\PGL_r)$, as in \cite{MP}. We also denote by $\left[\mathop{\oplus}\limits_{l=1}^{a_t}  \rho_{t,l}\right]$ the class of the representation in the symmetric product $\Sym^{a_t} 
 \big(\tilde\frX^{\irr}_{r_t}\big)$. Note that $\frX_r^{\{r\}}=\frX_r^{\irr}$ is the set of irreducible representations, and 
$\frX_r^{\{1^r\}}$ is the set of completely reducible representations, that is, those which are direct sum of one-dimensional representations.

In particular, for our rank $4$ case, we have
 \begin{equation}\label{eqn:78}
  [\frX_4]= [\frX_4^{irr}]+ \big[\frX_4^{\{3,1\}}\big]+  \big[\frX_4^{\{2^2\}}\big]+  \big[\frX_4^{\{2,1^2\}}\big]+\big[\frX_4^{\{1^4\}}\big]. 
  \end{equation}

\medskip

In order to study each of the previous strata, we consider the auxiliary varieties
  \begin{align*}
  \tilde\frX_r^{\pi} &=\big\{ \rho=\mathop{\oplus}\limits_{t=1}^s \left[\mathop{\oplus}\limits_{l=1}^{a_t}  \rho_{t,l} \right] \, \left| \,  
 \rho_{t,l} \in \tilde\frX^{\irr}_{r_t} \right.\big\}  
 = \prod_{t=1}^s \Sym^{a_t} 
 \big(\tilde\frX^{\irr}_{r_t}\big) \\
 \Sym^{a}_0\big(\tilde\frX^{\irr}_{r}\big) & = \big\{ \rho=\left[\mathop{\oplus}\limits_{l=1}^{a}  \rho_{l} \right] \,\left|\,  \rho_{l} \in \tilde\frX^{\irr}_{r_t}, \det \rho = 1\big\} \subset \Sym^a \big(\tilde\frX^{\irr}_{r} \right.\big).
  \end{align*}

If we decompose a partition $\pi = \big\{r_1^{a_1},\ldots, r_s^{a_s}\big\}$ into the partition $\pi' = \big\{r_2^{a_2},\ldots, r_{s}^{a_{s}}\big\}$ of 
$r' = \sum\limits_{i=2}^{s} a_ir_i$ and $\big\{r_1^{a_1}\big\}$, we have a Zariski locally trivial fibration
$
	\frX_r^{\pi} \to \tilde\frX_{r'}^{\pi'}
$
whose fiber is $\Sym^{a_1}_0\big(\tilde\frX^{\irr}_{r_1}\big)$. Therefore, we get that
$$
	[\frX_r^{\pi}] = \big[\Sym^{a_1}_0\big(\tilde\frX^{\irr}_{r_1}\big)\big][\tilde\frX_{r'}^{\pi'}] = \big[\Sym^{a_1}_0\big(\tilde\frX^{\irr}_{r_1}\big)\big] \prod_{t=2}^s \big[\Sym^{a_t}\big(\tilde\frX^{\irr}_{r_t}\big)\big].
$$
% ANGEL: Si no es correcto en general, al menos lo es con r1 = 1, que es lo que se necesita.

\begin{remark}\label{rmk:symmetric-prod-motive}
In many cases, the previous factors can be easily computed using known results:
\begin{itemize}
	\item The motive of the symmetric product $\Sym^a(X)$ can be computed from the motive of $X$ by means of the plethystic exponential, $\PExp$ \cite{Feng}. To be precise, if the motive of $X$ is generated by the Lefschetz motive, we have the following \cite[Proposition 4.6]{Florentino-Nozad-Zamora:2019a}

$$
	\sum_{a = 0}^\infty \big[\Sym^a(X)\big] z^a = \PExp\big([X]z\big),
$$
where, at the right hand side, we see $[X]z$ as a polynomial in $\ZZ[q,z]$.
	\item If $r_1 = 1$, we have that $\Sym^{a}_0\big(\tilde\frX^{\irr}_{1}\big)$ are the totally reducible representations of rank $a$. By \cite[Prop.\ 5.2]{MP}, their motives are $q^{a-1}$. In particular, $\big[\Sym^{1}_0\big(\tilde\frX^{\irr}_{1}\big)\big] = \big[\frX^{\irr}_{1}\big] = 1$.
%	\item Applying the previous formula to the partition $\pi = \big\{1,r\big\}$ of $r+1$, we obtain that
%	$$
%		\big[\tilde\frX_r^{\irr}\big] = \big[\Sym^{1}_0\big(\tilde\frX^{\irr}_{1}\big)\big] \big[\Sym^{1}\big(\tilde\frX^{\irr}_{1}\big)\big] = (q-1)\big[\frX_r^{\irr}\big].
%	$$ 
\end{itemize}
\end{remark}

With these observations at hand, we list the motives of each of the strata of (\ref{eqn:78}) of the rank $4$ case: 

\begin{itemize}
\item $[\frX_4^{\irr}]=[\frM_4^{\irr}]$ is computed in (\ref{eqn:frM4irr}). 
\item For the partition $\{3,1\}$, by the comments above, we have $\big[\frX_4^{\{3,1\}}\big]=[\tilde\frX^{\irr}_3]$.
The motive of this space is given in \cite[Cor.\ 10.3]{MP} (\cite[Prop.\ 10.1]{MP} specifies 
which strata correspond to irreducible representations, hence the polynomials
therein $P_0, P_5, P_6$ should be omitted). Let us denote for simplicity $P_1= (q-1)(q^4+4q^3-3q^2-15q+12)$,
$P_2= (q-1)(q^4 +2q^3-3q^2- q+4)$, $P_3=(q-1)(q^2-3q+3)$,
$P_4=(q-1)(q^2-q+1)$. Then, we have the following:

\begin{itemize}
\item If $n,m\equiv 1,5 \pmod 6$, then
\begin{align*}
	[\tilde\frX^{\irr}_3] &=
   \frac1{36} (m-1)(m-2)(n-1)(n-2) P_1 \\
   &\quad +
   \frac16(n-1)(m-1)(n+m-4) P_3.
\end{align*}
\item If $n\equiv 2,4 \pmod 6$ and $m\equiv 1,5\pmod 6$, then
\begin{align*}
	[\tilde\frX^{\irr}_3] &= 
   \frac1{36} (m-1)(m-2)(n-1)(n-2) P_1 \\
   & \quad +\frac16(n-1)(m-1)(n+m-4) P_3.
\end{align*}
\item If $n\equiv 3 \pmod 6$ and $m\equiv 1,5\pmod 6$, then
\begin{align*}
	[\tilde\frX^{\irr}_3] &= 
   \frac1{36} (m-1)(m-2)n(n-3) P_1+ \frac16(m-1)(m-2) P_2 \\
   &\quad + \frac16(m-1)(mn+n^2-5n-m-2) P_3 + (m-1) P_4.
\end{align*}
\item If $n\equiv 0 \pmod 6$ and $m\equiv 1,5\pmod 6$, then 
\begin{align*}
[\tilde\frX^{\irr}_3] &= 
   \frac1{36} (m-1)(m-2)n(n-3) P_1+ \frac16(m-1)(m-2) P_2 \\
   & \quad + \frac16(m-1)(mn+n^2-5n-m-2) P_3 + (m-1) P_4.
\end{align*}
\item If $n\equiv 2,4 \pmod 6$ and $m\equiv 3 \pmod 6$, then 
\begin{align*}
[\tilde\frX^{\irr}_3] &= 
   \frac1{36}m (m-3) (n-1)(n-2) P_1+ \frac16(n-1)(n-2) P_2 \\
   & \quad +
   \frac16(n-1)(mn+m^2-n-5m-2) P_3 + (n-1) P_4.
\end{align*}
   \end{itemize}
Here note that $\gcd(n,m)=1$, and that we can interchange $n,m$ if necessary.
  
\item The partition $\{2^2\}$ corresponds to representations of the form $\rho=\rho_1\oplus \rho_2$, where
$\rho_1,\rho_2$ are irreducible representations of rank $2$ and $\det \rho_2=(\det \rho_1)^{-1}$.
%Denote $\tilde\frX_2^{\irr}=\frX^{\irr}(\G_{n,m},\GL_2)$ and $\bar\frX_2^{\irr}=\frX^{\irr}(\G_{n,m},\PGL_2)$.
There is a locally trivial fibration $\CC^* \to \tilde\frX_2^{\irr}\to \bar \frX_2^{\irr}$.
It gives a fibration
 $$
  \CC^* \to \frX_4^{\{2^2\}}\to  \Sym^2 \bar\frX_2^{\irr}
 $$
via the map $\rho\mapsto ([\rho_1],[\rho_2])$. The fiber is isomorphic to $\CC^*$ through the parametrization $\lambda \mapsto (\lambda\rho_1,\lambda^{-1}\rho_2)$. Notice that the symmetric product is taken via the action of $\ZZ_2$ 
 on $(\bar\frX_2^{\irr})^2$ by $(\rho_1,\rho_2) \mapsto 
 (\rho_2,\rho_1)$. This acts on
 the fiber of the fibration via $\lambda\mapsto \lambda^{-1}$. Therefore, by the formula \cite[Prop.\ 2.6]{lomune}, we have
 $$
 \big[\frX_4^{\{2^2\}}\big]= [\CC^*]^+ \cdot [ (\bar\frX_2^{\irr})^2]^+ + [\CC^*]^- \cdot [ (\bar\frX_2^{\irr})^2]^-\, .
 $$
 Here, for $X$ an algebraic variety with a $\ZZ_2$-action, we denote $[X]^+=[X/\ZZ_2]$ and $[X]^- = [X]-[X]^+$.
By \cite[Eqn.\ (1)]{lomune}, we have $ [\CC^*]^+=q$ and $ [\CC^*]^-=-1$. 

The motive of $\bar\frX_2^{\irr}$ is provided in \cite[Prop.\ 7.2]{MP} and it is given by
$$
	[\bar\frX_2^{\irr}] = \left\{\begin{array}{ll} \frac{(n-1)(m-1)}{4} (q-2) & \textrm{if } n,m \textrm{ odd,} \\
	\frac{(n-2)(m-1)}{4} (q-2)+\frac{m-1}2 (q-1) & \textrm{if } n \textrm{ even.} \end{array}\right.
$$

By \cite{GLM}, the symmetric product is an operation on the Grothendieck ring. By Remark \ref{rmk:symmetric-prod-motive}, if we see $\left[X\right]$ as a polynomial in $q$ denoted by $\left[X\right](q)$, we have the formula
 $$
 	\Sym^2(\left[X\right](q)) = \frac12 \left[X\right](q^2) + \frac12 \left[X\right](q)^2,
 $$
Hence for $n,m$ odd we have
  \begin{align*}
 \big[ (\bar\frX_2^{\irr})^2\big]^+ &= \big[ \Sym^2(\bar\frX_2^{\irr})\big] =
  \frac{(m-1)^2(n-1)^2}{32}(q-2)^2+ \frac{(m-1)(n-1)}{8}(q^2-2) , \\
 \big[ (\bar\frX_2^{\irr})^2\big]^- &=  \big[ (\bar\frX_2^{\irr})^2\big] -  \big[ (\bar\frX_2^{\irr})^2\big]^+  \\ &=   
  -\frac{(n-1)^2(m-1)^2}{32}(q- 2)^2 + \frac{(n-1)(m-1)}{8}(-q^2 + 2q - 2) ,\\ 
  \big[\frX_4^{\{2^2\}}\big] &= \frac{(m-1)^2(n-1)^2}{32} (q^3 - 3q^2 + 4) + \frac{(m-1)(n-1)}{8}(q^3 + q^2 - 4q + 2),
 \end{align*}  
and for $n$ even and $m$ odd,
  \begin{align*}
 \big[ (\bar\frX_2^{\irr})^2\big]^+  =& \, \big[ \Sym^2(\bar\frX_2^{\irr})\big] =  \frac{(m-1)^2(n-1)^2}{32}(q-2)^2+ \frac{(m-1)(n-1)}{8}(q^2-2) \\ 
 & + \frac{(m-1)^2}{8}(q-1)^2 + \frac{m-1}{4}(q^2-1) + \frac{(n-1)(m-1)^2}{8}(q^2-3q+2)
  , \\
 \big[ (\bar\frX_2^{\irr})^2\big]^- =& -\frac{(n-1)^2(m-1)^2}{32}(q- 2)^2 + \frac{(n-1)(m-1)}{8}(-q^2 + 2q - 2)
   \\
   &  - \frac{(m-1)^2}{8}(q-1)^2 - \frac{m-1}{4}(q-1)^2 - \frac{(n-1)(m-1)^2}{8}(q^2-3q+2) ,\\ 
  \big[\frX_4^{\{2^2\}}\big] =& \, \frac{(m-1)^2(n-1)^2}{32}(q^3 - 3q^2 + 4)+ \frac{(m-1)(n-1)}{8}(q^3 + q^2 - 4q + 2) \\ & + \frac{(m-1)^2}{8}(q^3 - q^2 - q + 1) + \frac{m-1}{4}(q^3 + q^2 - 3q + 1) \\
  & + \frac{(n-1)(m-1)^2}{8}(q^3 - 2q^2 - q + 2).
 \end{align*}  

\item For the partition $\big\{ 2,1^2 \big\}$, we have that
$$
	\big[\frX_4^{\{2,1^2\}}\big] = \big[\Sym^{2}_0\big(\tilde\frX^{\irr}_{1}\big)\big]\, [\tilde\frX_{2}] = q[\tilde\frX_{2}].
$$

\item The partition $\big\{1^4\big\}$, we directly have that
 $$
 \big[\frX_4^{\{1^4\}}\big]=q^3.
 $$
\end{itemize}

Putting all these contribution together, this completes the proof of Theorem \ref{thm:main}.

%\end{document}

\end{document}